\documentclass[10pt]{amsart}

\usepackage[latin1]{inputenc}
\usepackage[english]{babel}

\usepackage{color}

\usepackage{indentfirst}
\usepackage{amssymb}
\usepackage{amsthm}

\newcommand{\con}{\mathfrak c}

\newcommand{\vf}{\varphi}
\newcommand{\sm}{\setminus}

\def\cF{{{\mathcal F}}}

\def\natnums{\mathbb N}

\def\N{\natnums}

\newcommand{\Ba}{{\rm Ba}}
\newcommand{\Bo}{{\rm Bo}}

\newcommand{\impli}{\Rightarrow}
\newcommand{\Nat}{\mathbb{N}}

\newcommand{\erre}{\mathbb{R}}

\newcommand{\JJ}{{\mathcal J}}

\newcommand{\ol}{\overline}

\newcommand{\too}{2^{\omega_1}}

\newtheorem{theo}{Theorem}[section]
\newtheorem{lem}[theo]{Lemma}
\newtheorem{pro}[theo]{Proposition}
\newtheorem{cor}[theo]{Corollary}
\newtheorem{defi}[theo]{Definition}

\newtheorem{thm}[theo]{Theorem}

\newtheorem{lemma}[theo]{Lemma}

\theoremstyle{definition}

\theoremstyle{remark}
\newtheorem{remark}[theo]{Remark}
\numberwithin{equation}{section}

\def\epsilon{\varepsilon}

\providecommand{\MR}{\relax\ifhmode\unskip\space\fi MR }

\providecommand{\href}[2]{#2}

\title{Measurability in $C(2^\kappa)$ and Kunen cardinals}

\author{ A. Avil\'{e}s}
\address{Departamento de Matem\'{a}ticas\\
Facultad de Matem\'{a}ticas\\ Universidad de Murcia\\ 30100 Espinardo (Murcia)\\
Spain} \email{avileslo@um.es}

\author{G. Plebanek}
\address{Instytut Matematyczny\\ Uniwersytet Wroc\l awski\\ Pl.\ Grunwaldzki 2/4\\
50-384 Wroc\-\l aw\\ Poland} \email{grzes@math.uni.wroc.pl}

\author{J. Rodr\'{i}guez}
\address{Departamento de Matem\'{a}tica Aplicada\\
Facultad de Inform\'{a}tica\\ Universidad de Murcia\\ 30100 Espinardo (Murcia)\\
Spain} \email{joserr@um.es}



\thanks{A.~Avil\'{e}s and J.~Rodr\'{i}guez were supported by MEC and FEDER (Project MTM2008-05396)
and Fundaci\'{o}n S\'{e}neca (Project 08848/PI/08). A. Avil\'{e}s was 
supported by {\em Ramon y Cajal} contract (RYC-2008-02051) and an FP7-PEOPLE-ERG-2008 action.
G. Plebanek was partially supported by MNiSW Grant N N201 418939 (2010--2013)}

\begin{document}

\begin{abstract}
A cardinal $\kappa$ is called a Kunen cardinal if the $\sigma$-algebra on~$\kappa \times \kappa$
generated by all products $A\times B$, where $A,B \subset \kappa$, coincides
with the power set of~$\kappa\times \kappa$. For any cardinal~$\kappa$,
let $C(2^\kappa)$
be the Banach space of all continuous real-valued functions on the Cantor cube~$2^\kappa$.
We prove that $\kappa$ is a Kunen cardinal if and only if the Baire $\sigma$-algebra on~$C(2^\kappa)$ for
the pointwise convergence topology coincides with the Borel $\sigma$-algebra 
on~$C(2^\kappa)$ for the norm topology. Some other links between Kunen cardinals
and measurability in Banach spaces are also given.
\end{abstract}

\maketitle

\section{Introduction}

In every completely regular topological space~$T$ there are two natural $\sigma$-algebras:
the Borel $\sigma$-algebra $\Bo(T)$ generated by all open sets and, usually much smaller, the Baire
$\sigma$-algebra $\Ba(T)$ generated by all continuous
real-valued functions on~$T$. For a Banach space~$X$, we always have
$$
    \Ba(X_w) \subset \Bo(X_w) \subset \Bo(X)=\Ba(X)
$$
where $X_w$ stands for $X$ equipped with its weak topology.
Moreover, for the Banach space $C(K)$ of all continuous real-valued functions
on a compact space~$K$, other $\sigma$-algebras appear:
$$
\begin{array}{c c c c c}
\Ba(C_p(K)) & \subset & \Bo(C_p(K)) & \mbox{ } & \mbox{ } \\
 \cap & \mbox{ } & \cap & \mbox{ } & \mbox{ } \\
\Ba(C_w(K)) & \subset & \Bo(C_w(K)) & \subset & \Bo(C(K))
\end{array}
$$
where $C_p(K)$ (resp. $C_w(K)$)
stands for $C(K)$ equipped with the pointwise convergence (resp. weak) topology.
It is well-known that all these $\sigma$-algebras coincide for separable
Banach spaces. For nonseparable Banach spaces some of the inclusions above might be
strict and the equalities between these $\sigma$-algebras are closely related
to several interesting properties of~$X$ and~$K$, see e.g. \cite{bur-pol-1,bur-pol-2,edg-J,edg1,mar-pol-1,mar-pol-2,tal9}.

The first example of a nonseparable Banach space~$X$ for which
$\Ba(X_w)=\Bo(X)$ was given by Fremlin~\cite{fre6} showing that
such equality holds for $X=\ell^1(\omega_1)$. For any cardinal~$\kappa$, Fremlin proved that
the equality 
$$
	\Ba(\ell^1(\kappa)_w)=\Bo(\ell^1(\kappa))
$$
is equivalent to saying that 
\begin{equation}\label{equation:Kunen}
	\mathcal{P}(\kappa \times \kappa)=\mathcal{P}(\kappa)\otimes \mathcal{P}(\kappa)
\end{equation}
(i.e. the power set of $\kappa \times \kappa$ coincides with the $\sigma$-algebra 
on~$\kappa \times \kappa$ generated by all products $A\times B$, where $A,B \subset \kappa$). 
From now on we shall say that a cardinal~$\kappa$ is a {\em Kunen cardinal} if~\eqref{equation:Kunen} holds.
This notion has its origin in a problem posed by Ulam \cite{Ulam} and was investigated by Kunen in his doctoral dissertation \cite{Kunen68}. Let us mention that:
\begin{enumerate}
\item[(i)] any Kunen cardinal is less than or equal to~$\mathfrak{c}$; 
\item[(ii)] $\omega_1$ is a Kunen cardinal; 
\item[(iii)] $\con$ is a Kunen cardinal under Martin's axiom, while it is relatively consistent that $\con$ 
is not a Kunen cardinal. 
\end{enumerate}
Kunen cardinals have been also considered by Talagrand~\cite{tal10} 
in connection with measurability properties of Banach spaces, and in a paper by Todorcevic \cite{Todorcevic} on universality properties of $\ell_\infty/c_0$, where the reader can find more accurate historical remarks on this topic. 

In this paper we focus on the Banach space $C(2^\kappa)$ for a cardinal~$\kappa$ and prove that
the equality 
$$
	\Ba(C_p(2^\kappa))=\Bo(C(2^\kappa))
$$
holds if and only if $\kappa$ is a Kunen cardinal (Theorem~\ref{thm:Main}). This 
extends Fremlin's aforementioned result, since $C(2^\kappa)$ contains~$\ell^1(\kappa)$
isomorphically. The picture of coincidence of $\sigma$-algebras on~$C(2^\kappa)$ is then the following:
\begin{itemize}

\item[(a)] $\Bo(C_p(2^\kappa)) = \Bo(C(2^\kappa))$ for any $\kappa$, since
$C(2^\kappa)$ admits a pointwise Kadec equivalent norm, 
see e.g. \cite[VII.1.10]{dev-alt-J} and~\cite{edg1}.

\item[(b)] $\Ba(C_p(2^\kappa)) = \Ba(C_w(2^\kappa))$ if and only if $\kappa\leq\mathfrak{c}$.
Indeed, the ``if'' follows from the fact that any Radon probability on~$2^{\mathfrak{c}}$
admits a uniformly distributed sequence (cf. \cite[491Q]{freMT-4}). On the other hand, if
$\kappa>\mathfrak{c}$ then $2^\kappa$ is nonseparable and so
the standard product measure on~$2^\kappa$ cannot be $\Ba(C_p(2^\kappa))$-measurable
(cf. \cite[Proposition~3.6]{rod-ver}).

\item[(c)] $\Ba(C_p(2^\kappa)) = \Bo(C(2^\kappa))$ if and only if $\kappa$ is a Kunen cardinal. 

\end{itemize} 

The paper is organized as follows. Section~\ref{section:Main} is entirely devoted
to prove statement~(c) (Theorem~\ref{thm:Main}). The proof is self-contained and rather technical. 

In Section~\ref{section:omega1} we single out a certain topological property of a compact space $K$
which guarantees that $\Ba(C_p(K)) = \Bo(C_p(K))$ (Corollary \ref{1:2}).
That property holds for $K=2^{\omega_1}$ and this gives a more direct proof of the equality
$\Ba(C_p(2^{\omega_1})) = \Bo(C(2^{\omega_1}))$ which relies on statement~(a) above. 

In Section~\ref{section:Norms}
we show that a Banach space~$X$ admits a non $\Ba(X_w)$-measurable equivalent norm 
whenever $X$ has a biorthogonal system of non Kunen cardinality
(Theorem~\ref{renormeBABS}): this applies to $C(2^{\kappa})$ and $\ell^1(\kappa)$ provided that $\kappa$ is not Kunen.

\subsection*{Terminology}

For any $n\in \Nat$ we write $2^n:=\{0,1\}^n$. As usual, $\omega_1$ denotes the first uncountable ordinal
and $\mathfrak{c}$ is the cardinality of the continuum. All our topological spaces are assumed to be Hausdorff.
Given a measurable space $(Y,\Sigma)$ and $S \subset Y$, the {\em trace of~$\Sigma$ on~$S$}
is the $\sigma$-algebra on~$S$ defined by $\{S\cap A:A\in \Sigma\}$.

Given any set $\Gamma$, we write $\mathcal{P}(\Gamma)$ to denote the power set of~$\Gamma$. The symbol
$|\Gamma|$ stands for the cardinality of~$\Gamma$. The $\sigma$-algebra 
on~$\Gamma^2=\Gamma\times \Gamma$ generated by all products $A\times B$, where $A,B \subset \Gamma$, is
denoted by $\mathcal{P}(\Gamma)\otimes \mathcal{P}(\Gamma)$. For any $U \subset \Gamma$, the characteristic 
function $1_U:\Gamma \to \{0,1\}$ is defined by $1_U(\gamma)=1$ if $\gamma\in U$, $1_U(\gamma)=0$ if $\gamma\in U$.
We denote by $2^\Gamma$ the Cantor cube, i.e. the set of all $\{0,1\}$-valued
functions on~$\Gamma$, which becomes a compact space when equipped with the pointwise convergence topology.
$\mathcal{P}(\Gamma)$ and~$2^\Gamma$ can be identified via $U \mapsto 1_U$.

Given a set $E$ and $\cF \subset \erre^E$, we write $\sigma(\cF)$ to denote the
$\sigma$-algebra on~$E$ generated by~$\cF$ (i.e. the smallest one for which every $f\in \cF$ is measurable).
It is well-known that if $E$ is a locally convex space then $\Ba(E_w)=\sigma(E')$, where 
$E_w$ stands for~$E$ equipped with its weak topology and $E'$ is the (topological) dual of~$E$, 
see \cite[Theorem~2.3]{edg-J}. In particular, we have:
\begin{itemize}
\item[(i)] $\Ba(C_p(K))=\sigma(\{\delta_t:t\in K\})$ for every compact space~$K$, where
$\delta_t$ denotes the Dirac delta at $t\in K$.
\item[(ii)] $\Ba(X_w)=\sigma(X^*)$ for every Banach space~$X$ (with dual $X^*$).
\end{itemize}
In view of~(ii) and the Hahn-Banach theorem, if $Y$ is a closed subspace of a 
Banach space~$X$, then the trace of~$\Ba(X_w)$ on~$Y$ is exactly~$\Ba(Y_w)$.

\section{The main result}\label{section:Main}

The aim of this section is to prove that the equality
$\Ba(C_p(2^\Gamma))=\Bo(C(2^\Gamma))$ is equivalent to saying that $|\Gamma|$ is a Kunen
cardinal (Theorem~\ref{thm:Main} below). The proof is split
into several lemmas for the convenience of the reader. Throughout this section~$\Gamma$ is a fixed infinite set.

\begin{lem}\label{lem:MoreThanContinuum}
Let $A \in \mathcal{P}(\Gamma) \otimes \mathcal{P}(\Gamma)$. Define an equivalence relation 
$\approx$ on~$\Gamma$ by saying that $\gamma \approx \gamma'$ if and only if, for each $\delta \in \Gamma$, we have
$$
	(\delta,\gamma)\in A \ \Leftrightarrow \ (\delta,\gamma') \in A
	\quad \mbox{and} \quad
	(\gamma,\delta)\in A \ \Leftrightarrow \ (\gamma',\delta) \in A.
$$
Then $\approx$ has at most $\mathfrak{c}$ many equivalence classes.
\end{lem}
\begin{proof} Take $B_n \subset \Gamma$, $n\in \Nat$, such that
$A$ belongs to the $\sigma$-algebra $\mathcal{A}_0$ on~$\Gamma^2$ generated by the sequence $(B_{2m}\times B_{2m-1})_{m\in \Nat}$.
Define an equivalence relation $\sim$ on~$\Gamma$ by 
$$
	\gamma \sim \gamma' \ \Leftrightarrow  \ 1_{B_n}(\gamma) = 1_{B_n}(\gamma') \mbox{ for all }n\in \Nat. 
$$
Since there are at most $\mathfrak{c}$ distinct sequences of the form $(1_{B_n}(\gamma))_{n\in \Nat}\in 2^\Nat$, 
the relation $\sim$ has at most $\mathfrak{c}$ many equivalence classes.
Let $\mathcal{A}_1$ be the family made up of all $C \in \mathcal{A}_0$
such that, for each $\gamma\sim \gamma'$ and $\delta\sim\delta'$, we have
$$
	(\gamma,\delta)\in C \ \Leftrightarrow \ (\gamma',\delta')\in C.
$$
Clearly $\mathcal{A}_1$ is a $\sigma$-algebra containing $B_{2m}\times B_{2m-1}$
for all $m\in \Nat$, hence $\mathcal{A}_0=\mathcal{A}_1$ and so $A\in \mathcal{A}_1$. In particular, we have
$\gamma \approx \gamma'$ whenever $\gamma \sim \gamma'$. It follows that 
the relation $\approx$ has at most $\mathfrak{c}$ many equivalence classes as well.
\end{proof}

Part~(ii) of the following lemma is well-known, see \cite{Kunen68}.

\begin{lem}\label{lem:MoreNoKunen}
Let $\Omega=\{(\gamma_1,\gamma_2)\in \Gamma^2: \gamma_1\neq \gamma_2\}$ and let 
$\Sigma$ be the trace of $\mathcal{P}(\Gamma)\otimes \mathcal{P}(\Gamma)$ on~$\Omega$. Then:
\begin{enumerate}
\item[(i)] $|\Gamma|$ is a Kunen cardinal if and only if $\Sigma = \mathcal{P}(\Omega)$.
\item[(ii)] If $|\Gamma|>\mathfrak{c}$, then $|\Gamma|$ is not a Kunen cardinal.
\end{enumerate}
\end{lem}
\begin{proof} We distinguish two cases:

\smallskip
{\sc Case $|\Gamma| \leq \mathfrak{c}$.} We can assume without loss of generality that $\Gamma \subset \erre$. 
For each $U\subset \Gamma$, we have
\begin{equation}\label{equation:diagonalsubsets}
	\{(\gamma,\gamma):\gamma \in U\}=
	\bigcap_{n\in \Nat} \bigcup_{q\in \mathbb{Q}} 
	\left(U\cap \Bigl(q-\frac{1}{n},q+\frac{1}{n}\Bigr)\right)^2 \in \mathcal{P}(\Gamma) \otimes \mathcal{P}(\Gamma).
\end{equation}
In particular, we get $\Omega \in \mathcal{P}(\Gamma) \otimes \mathcal{P}(\Gamma)$ and so $\Sigma \subset \mathcal{P}(\Gamma) \otimes \mathcal{P}(\Gamma)$.
Suppose now that $|\Gamma|$ is not a Kunen cardinal. If $A\subset \Gamma^2$ is any set not belonging to $\mathcal{P}(\Gamma) \otimes \mathcal{P}(\Gamma)$,
then $A\cap \Omega \not \in \Sigma$ because 
\eqref{equation:diagonalsubsets} implies that $A\setminus \Omega \in \mathcal{P}(\Gamma) \otimes \mathcal{P}(\Gamma)$.

\smallskip
{\sc Case $|\Gamma|> \mathfrak{c}$.} Let $\equiv$ be an equivalence relation on~$\Gamma$ for which all equivalence classes
are infinite and have cardinality less than or equal to~$\mathfrak{c}$.
We shall check that the set 
$$
	W:=\{(\gamma_1,\gamma_2)\in \Omega: \, \gamma_1 \equiv \gamma_2\}
$$
does not belong to~$\Sigma$. Suppose if possible otherwise. Then there is $A\in \mathcal{P}(\Gamma)\otimes\mathcal{P}(\Gamma)$
such that $A\cap \Omega = W$. Let $\approx$ be the equivalence relation on~$\Gamma$ 
induced by~$A$ as defined in Lemma~\ref{lem:MoreThanContinuum}. Since $|\Gamma|>\mathfrak{c}$, an appeal to
Lemma~\ref{lem:MoreThanContinuum} ensures the existence of $E \subset \Gamma$ with $|E|>\mathfrak{c}$ such that
$\gamma \approx \gamma'$ whenever $\gamma,\gamma'\in E$. Given distinct $\gamma,\gamma'\in E$
we can find $\delta \in \Gamma \setminus \{\gamma,\gamma'\}$ with $\delta \equiv \gamma$.
Then $(\delta,\gamma)\in W = A\cap \Omega$ and the fact that $\gamma \approx \gamma'$ implies
that $(\delta,\gamma')\in A\cap \Omega=W$, hence $\gamma \equiv \gamma'$. This means that $E$ is contained
in some equivalence class of~$\equiv$, which has cardinality less than or equal to~$\mathfrak{c}$.
This contradiction finishes the proof.
\end{proof}

From now on we denote by $\mathfrak I$ the family of all closed nonempty intervals of~$\mathbb{R}$.

\begin{defi} Let $n\in \Nat$.
\begin{enumerate}
\item[(i)] A function $\tau: 2^n \to \mathfrak{I}$ is called a {\em type}
(or an {\em $n$-type}). 
\item[(ii)] Let $\tau$ be an $n$-type. We say that $f\in C(2^\Gamma)$ {\em has type~$\tau$} if there exist 
$\gamma_1,\dots,\gamma_{n}\in\Gamma$ such that 
$$
	f(x) \in \tau(x_{\gamma_1},\dots,x_{\gamma_{n}}) \quad \mbox{for every }x\in 2^\Gamma.
$$
We denote by $Y_\tau$ the set of all $f\in C(2^\Gamma)$ having type~$\tau$.
\end{enumerate}
\end{defi}

\begin{lem}\label{lem:MeasurabilityType}
If $|\Gamma|\leq \mathfrak{c}$, then $Y_\tau$ belongs to $\Ba(C_p(2^\Gamma))$ for every type $\tau$.
\end{lem}
\begin{proof}
Since $|\Gamma|\leq \mathfrak{c}$, we can suppose that $\Gamma$ is a subset of the Cantor set $\Delta = 2^\mathbb{N}$.
We write $\gamma = (\gamma[m])_{m\in\N}$ when we express $\gamma\in \Delta$ as a sequence of $0$'s and $1$'s. 
For each $m\in \Nat$, we consider 
$$
	\Gamma_m := \{\gamma\in \Delta : \, \gamma[k]=0 \mbox{ for all }k>m\}.
$$
Observe that $\bigcup_{m\in\mathbb{N}}\Gamma_m$ is countable and so we 
can suppose without loss of generality that $\bigcup_{m\in\mathbb{N}}\Gamma_m \subset \Gamma$. For each $m\in \Nat$,
let 
$$
	K_m := \{x\in 2^\Gamma: \ x_\gamma = x_\delta \mbox{ whenever }\gamma,\delta \in \Gamma
	\mbox{ satisfy } \gamma[k] = \delta[k] \mbox{ for all }k\leq m\}.
$$
Note that $K_m$ is finite. Indeed, it is easy to check that
$K_m=\{x^\sigma:\sigma \in 2^{2^m}\}$, where $x^\sigma \in 2^\Gamma$ is defined by 
$x^\sigma(\gamma):=\sigma((\gamma[1],\dots,\gamma[m]))$ for all $\gamma\in \Gamma$.

Let $n\in \Nat$ be such that $\tau$ is an $n$-type. The set
$$
	A:=
	\bigcap_{m\in \Nat} \,
	\bigcup_{\gamma_1^m,\dots,\gamma_n^m\in\Gamma_m} \,
	\bigcap_{x\in K_m}
	\{f\in C(2^\Gamma) : \, f(x)\in \tau(x_{\gamma_1^m},\dots,x_{\gamma_n^m})\}
$$
belongs to $\Ba(C_p(2^\Gamma))$. So, in order to prove that $Y_\tau\in\Ba(C_p(2^\Gamma))$ it is enough to check that
$Y_\tau = A$.

We first prove $Y_\tau \subset A$. Take $f\in Y_\tau$. Then there exist $\gamma_1,\dots,\gamma_n\in\Gamma$ such that $f(x)\in\tau(x_{\gamma_1},\ldots,x_{\gamma_n})$ for every $x\in 2^\Gamma$. Given $m\in \Nat$ and $i\in \{1,\dots,n\}$, we can 
choose $\gamma_i^m\in \Gamma_m$ such that $\gamma_i^m[k] = \gamma_i[k]$ for all $k\leq m$. For each $x\in K_m$ we have
$x_{\gamma_i^m} = x_{\gamma_i}$ and hence $f(x)\in \tau(x_{\gamma_1^m},\ldots,x_{\gamma_n^m})$. 
Therefore, $f\in A$.

We now prove $A \subset Y_\tau$. Take $f\in A$. We can consider the function 
$\tilde{f}\in C(2^{\Delta})$ given by $\tilde{f}(x) := f(x|_\Gamma)$. For each $m \in \Nat$, set 
$$
	\tilde{K}_m := \{x\in 2^\Delta: \ x_\gamma = x_\delta \mbox{ whenever }\gamma,\delta \in \Delta
	\mbox{ satisfy } \gamma[k] = \delta[k] \mbox{ for all }k\leq m\},
$$
$$
	P_m = \{(\gamma_1,\ldots,\gamma_n)\in \Delta^n : \, 
	\tilde{f}(x)\in \tau(x_{\gamma_1},\dots,x_{\gamma_n})
	\mbox{ for all } x\in \tilde{K}_m\}.
$$
Observe that $P_m \neq \emptyset$ because $f\in A$ and
$x|_\Gamma\in K_m$ whenever $x\in \tilde{K}_m$. It is easy to check that, for each $x\in \tilde{K}_m$, the set
$ \{(\gamma_1,\ldots,\gamma_n)\in \Delta^n : \tilde{f}(x)\in \tau(x_{\gamma_1},\dots,x_{\gamma_n})\}$ is closed,
hence $P_m$ is compact. Now, since $P_m\supset P_{m+1}$ for all $m\in \Nat$, 
we can pick $(\delta_1,\ldots,\delta_n)\in \bigcap_{m\in\N}P_m$. Then 
$\tilde{f}(x)\in \tau(x_{\delta_1},\ldots,x_{\delta_n})$ for every $x\in \bigcup_{m \in \Nat}\tilde{K}_m$.

We claim that $\bigcup_{m\in \Nat}\tilde{K}_m$ is dense in $2^\Delta$. 
Indeed, fix $z\in 2^\Delta$ and take a finite set
of coordinates $\{\gamma_1,\dots,\gamma_p\}\subset \Delta$. Choose
$m\in \Nat$ large enough such that $(\gamma_i[1],\dots,\gamma_i[m])\neq (\gamma_j[1],\dots,\gamma_j[m])$
whenever $i\neq j$. Then the element $x\in 2^\Delta$ defined by
$$
	x_\gamma:=
	\begin{cases}
	z_{\gamma_i} & \text{if $\gamma[k]=\gamma_i[k]$ for all $k\leq m$},\\
	0 & \text{otherwise},
	\end{cases}
$$
belongs to~$\tilde{K}_m$ and satisfies $x_{\gamma_i}=z_{\gamma_i}$ for every~$i$. This proves
the claim.

It follows that $\tilde{f}(x)\in \tau(x_{\delta_1},\ldots,x_{\delta_n})$ for every $x\in 2^\Delta$.
We choose an arbitrary $\xi\in \Gamma$ and, for each $i\in \{1,\dots,n\}$, we define $\gamma_i := \delta_i$ if $\delta_i\in\Gamma$ and 
$\gamma_i := \xi$ if $\delta_i\not\in\Gamma$. 
We claim that $f(x)\in \tau(x_{\gamma_1},\dots,x_{\gamma_n})$ for every $x\in 2^\Gamma$. Indeed, given any $x\in 2^\Gamma$,
we can select $z\in 2^\Delta$ such that $z|_\Gamma = x$ and $z_{\delta_i} = x_\xi$ whenever $\delta_i\not\in \Gamma$, so that
$$
	f(x) = \tilde{f}(z)\in \tau(z_{\delta_1},\ldots,z_{\delta_n}) = \tau(x_{\gamma_1},\ldots,x_{\gamma_n}),
$$
as claimed. This shows that $f\in Y_\tau$ and the proof is over.
\end{proof}

The proof of the key Lemma~\ref{Mainlemma} is rather technical and will be given later
(Subsection~\ref{subsection:lemma}). In order to state that lemma we first need some definitions. 
From now on, the ``coordinates'' of any $\gamma \in \Gamma^n$, $n\in \Nat$, are denoted
by $\gamma_1,\dots,\gamma_n$, that is, we write $\gamma=(\gamma_1,\dots,\gamma_n)$.

\begin{defi}
Let $\tau$ be an $n$-type.
\begin{enumerate}
\item[(i)] We say that $\gamma,\delta\in\Gamma^n$ are {\em $\tau$-proximal} if
$$
	\tau(1_U(\gamma_1),\ldots,1_U(\gamma_n))\cap \tau(1_U(\delta_1),\ldots,1_U(\delta_n)) \neq\emptyset
$$
for every $U\subset\Gamma$.
\item[(ii)] We say that $A,B\subset \Gamma^n$ are {\em $\tau$-separated} if there exist no 
$\gamma\in A$ and $\delta\in B$ which are $\tau$-proximal.
\end{enumerate}
\end{defi}

\begin{defi}
Let $(Y,\Sigma)$ be a measurable space. We say that $U,V \subset Y$ are 
{\em $\Sigma$-separated} if there is $S\in \Sigma$ such that $U \subset S$ and $V \cap S=\emptyset$.
\end{defi}

\begin{lem}\label{Mainlemma}
Let $\tau$ be an $n$-type, $(Y,\Sigma)$ a measurable space and $\Phi:\Gamma^n \to \mathcal{P}(Y)$ a multifunction 
satisfying: 
\begin{itemize}
\item[(S)] For each $U\subset\Gamma$ and each closed set $I\subset \mathbb{R}$, the sets
$$
	\Phi\bigl(\{\gamma\in\Gamma^n : \, \tau(1_U(\gamma_1),\dots,1_U(\gamma_n))\subset I\}\bigr)
$$
$$
	\Phi\bigl(\{\gamma\in\Gamma^n : \, \tau(1_U(\gamma_1),\dots,1_U(\gamma_n))\cap I = \emptyset\}\bigr)
$$
are $\Sigma$-separated.
\end{itemize}
Suppose $|\Gamma|$ is a Kunen cardinal. If $A,B \subset \Gamma^n$ are $\tau$-separated, 
then $\Phi(A)$ and $\Phi(B)$ are $\Sigma$-separated.
\end{lem}

We write $C(2^\Gamma,2)$ to denote the subset of~$C(2^\Gamma)$ made up
of all $\{0,1\}$-valued functions, which can be identified with the algebra
${\rm Clop}(2^\Gamma)$ of all clopen subsets of~$2^\Gamma$ via
the bijection 
$$
		\psi: {\rm Clop}(2^\Gamma) \to C(2^\Gamma,2),
		\quad \psi(A):=1_A.
$$
The trace of $\Ba(C_p(2^\Gamma))$ on $C(2^\Gamma,2)$ is denoted by $\Ba(C_p(2^\Gamma,2))$.
Observe that $\{\psi^{-1}(E):E\in \Ba(C_p(2^\Gamma,2))\}$ is exactly
the $\sigma$-algebra on~${\rm Clop(2^\Gamma)}$ generated by all ultrafilters.
On the other hand, since $C(2^\Gamma,2)$ is norm discrete, 
the trace of $\Bo(C(2^\Gamma))$ on $C(2^\Gamma,2)$ is exactly $\mathcal{P}(C(2^\Gamma,2))$.

We now arrive at our main result:

\begin{thm}\label{thm:Main}
The following statements are equivalent:
\begin{enumerate}
\item[(i)] $|\Gamma|$ is a Kunen cardinal.
\item[(ii)] $\Ba(C_p(2^\Gamma)) = \Bo(C(2^\Gamma))$.
\item[(iii)] $\Ba(C_p(2^\Gamma,2)) = \mathcal{P}(C(2^\Gamma,2))$.
\item[(iv)] The $\sigma$-algebra on~${\rm Clop(2^\Gamma)}$ generated by all ultrafilters
is $\mathcal{P}({\rm Clop}(2^\Gamma))$.
\end{enumerate}
\end{thm}
\begin{proof}
(iii)$\Leftrightarrow$(iv) follows from the comments preceding the theorem.

\medskip
(i)$\impli$(ii).
Let us write $Y:=C(2^\Gamma)$ and $\Sigma := \Ba(C_p(2^\Gamma))$. Let $\Theta$ be an open subset of~$Y$ in the norm topology. 
We shall prove that $\Theta \in \Sigma$.

\smallskip
{\sc Step~1.} Fix an $n$-type $\tau$ and consider the multifunction $\Phi^\tau:\Gamma^n \to \mathcal{P}(Y)$ given by
$$
	\Phi^\tau(\gamma) := \left\{f\in Y: \, f(x) \in \tau(x_{\gamma_1},\dots,x_{\gamma_n})
	\mbox{ for all }x\in 2^\Gamma\right\} 
	\subset Y_\tau.
$$

We first observe that $\gamma,\delta\in\Gamma^n$ are $\tau$-proximal if and only if $\Phi^\tau(\gamma)\cap \Phi^\tau(\delta)\neq\emptyset$.
Indeed, the ``if'' part follows from the fact that 
$$
	f(1_U)\in \tau(1_U(\gamma_1),\ldots,1_U(\gamma_n))\cap \tau(1_U(\delta_1),\ldots,1_U(\delta_n))
$$ 
whenever $f\in \Phi^\tau(\gamma)\cap\Phi^\tau(\delta)$ and $U \subset \Gamma$.
Conversely, assume that $\gamma$ and $\delta$ are $\tau$-proximal. Then for each $U\subset \Gamma$ we can pick
$$
	t_U \in \tau(1_U(\gamma_1),\ldots,1_U(\gamma_n))\cap \tau(1_U(\delta_1),\ldots,1_U(\delta_n)).
$$
Let $W$ be the subset of~$\Gamma$ made up of all $\gamma_i$'s and $\delta_i$'s. Since $W$ is finite, 
the function $f: 2^\Gamma \to \erre$ given by $f(1_U):= t_{U\cap W}$ is continuous. Moreover,
since $1_U(\gamma_i)=1_{U\cap W}(\gamma_i)$ and $1_U(\delta_i)=1_{U\cap W}(\delta_i)$ for every $U \subset \Gamma$ and every~$i$,
we have $f\in \Phi^\tau(\gamma)\cap \Phi^\tau(\delta)$. Hence $\Phi^\tau(\gamma)\cap \Phi^\tau(\delta)\neq \emptyset$.

It follows at once that the following two subsets of~$\Gamma^n$ are $\tau$-separated:
$$
	A_\tau := \{\gamma\in\Gamma^n : \, \Phi^\tau(\gamma)\setminus \Theta \neq \emptyset\},
$$
$$
	B_\tau = \{\gamma\in\Gamma^n : \, \Phi^\tau(\gamma)\cap \Phi^\tau(A_\tau) = \emptyset\}.
$$
On the other hand, $Y_\tau \in \Sigma$ (by Lemmas~\ref{lem:MoreNoKunen} and~\ref{lem:MeasurabilityType}) and so,
for each $U\subset\Gamma$ and each closed set $I\subset \mathbb{R}$, the set 
$S_{(U,I)}:=\{f\in Y_\tau : f(1_U)\in I\}$
belongs to~$\Sigma$ and satisfies
$$
	\Phi^\tau\bigl(\{\gamma\in\Gamma^n : \, \tau(1_U(\gamma_1),\dots,1_U(\gamma_n))\subset I\}\bigr) 
	\subset S_{(U,I)},
$$
$$
	\Phi^\tau\bigl(\{\gamma\in\Gamma^n : \, \tau(1_U(\gamma_1),\dots,1_U(\gamma_n))\cap I = \emptyset\}\bigr)
	\cap S_{(U,I)} = \emptyset.
$$
An appeal to Lemma~\ref{Mainlemma} ensures that $\Phi^\tau(A_\tau)$ and $\Phi^\tau(B_\tau)$ are $\Sigma$-separated, that is,
there is $\Theta_\tau\in \Sigma$ such that $\Phi^\tau(B_\tau)\subset \Theta_\tau$ and $\Phi^\tau(A_\tau)\cap \Theta_\tau = \emptyset$.
Bearing in mind that $Y_\tau \in \Sigma$, we can assume further that $\Theta_\tau\subset Y_\tau$.

\smallskip
{\sc Step~2.} We write
$\mathfrak{I}_0$ to denote the (countable) family of all closed nonempty intervals of~$\mathbb{R}$ with rational endpoints. 
To finish the proof we shall check that
\begin{equation}\label{equation:theta}
	\Theta = \bigcup \{\Theta_\tau : \, \tau\text{ is a type with values in $\mathfrak{I}_0$} \}.
\end{equation}

On the one hand, for any $n$-type~$\tau$, we have $\Theta_\tau\subset Y_\tau\setminus \Phi^\tau(A_\tau)$. Moreover,
we have 
$Y_\tau\setminus \Phi^\tau(A_\tau) \subset \Theta$, because for each $f\in Y_\tau\setminus \Phi^\tau(A_\tau)$
there is some $\gamma \in \Gamma^n \setminus A_\tau$ such that $f \in \Phi^\tau(\gamma) \subset \Theta$.
Thus, the inclusion ``$\supset$'' in~\eqref{equation:theta} holds true. 

In order to prove the reverse inclusion, fix $f\in\Theta$. Since $\Theta$ is norm open, there is 
$\varepsilon>0$ such that $\|f-h\|_\infty\geq 2\varepsilon$ for every $h\in Y\setminus\Theta$.
By the continuity of~$f$ and the compactness of~$2^\Gamma$, we can find finitely many
basic clopen sets $C_i \subset 2^\Gamma$ such that $2^\Gamma=\bigcup_{i} C_i$ and 
the oscillation of~$f$ on each~$C_i$ is less than~$\varepsilon$. 
Thus, we can find a finite set $\{\gamma_1,\dots,\gamma_n\} \subset \Gamma$ and a type 
$\tau:2^n\to \mathfrak{J}_0$ such that: 
\begin{itemize}
\item[(a)] $\tau(p)$ has length less than $\varepsilon$ for every $p\in 2^n$, 
\item[(b)] $f(x)\in \tau(x_{\gamma_1},\ldots,x_{\gamma_n})$ for every $x\in 2^\Gamma$.
\end{itemize}
Condition (b) means that $f \in \Phi^\tau(\gamma)$, where $\gamma:=(\gamma_1,\dots,\gamma_n)\in \Gamma^n$.

We claim that $f\in \Theta_\tau$. Indeed, it suffices to check that
$\gamma \in B_\tau$, because in that case we would have
$f\in \Phi^\tau(\gamma) \subset \Phi^\tau(B_\tau)\subset \Theta_\tau$.
Our proof is by contradiction: suppose that $\gamma\not\in B_\tau$. 
Then there exists $\delta\in A_\tau$ such that $\Phi^\tau(\gamma)\cap \Phi^\tau(\delta)\neq\emptyset$.
Take $g\in \Phi^\tau(\gamma)\cap \Phi^\tau(\delta)$ and $h\in \Phi^\tau(\delta)\setminus\Theta$. By~(a)
we have:
$$
	\|u-v\|_\infty < \varepsilon 
	\quad
	\mbox{for every }u,v \in \Phi^\tau(\zeta) \mbox{ and every }\zeta \in \Gamma^n.
$$
Therefore, $\|f-g\|_\infty < \varepsilon$ (since $f,g\in \Phi^\tau(\gamma)$) 
and $\|g-h\|_\infty < \varepsilon$ (since $g,h\in \Phi^\tau(\delta)$). 
We conclude that $\|f-h\|_\infty< 2\varepsilon$, which contradicts the choice of~$\varepsilon$ because $h\not\in\Theta$.

\medskip
(ii)$\impli$(iii) is obvious.

\medskip
(iii)$\impli$(i). Let $\Omega:=\{(\gamma_1,\gamma_2)\in \Gamma^2: \gamma_1\neq \gamma_2\}$ be equipped
with the trace $\Sigma$ of the product $\sigma$-algebra $\mathcal{P}(\Gamma)\otimes \mathcal{P}(\Gamma)$. 
The function $H:\Omega \to C(2^\Gamma,2)$ given by 
$$
	H(\gamma_1,\gamma_2)(x) := x_{\gamma_1}(1-x_{\gamma_2})
$$
is $\Sigma$-$\Ba(C_p(2^\Gamma,2))$-measurable, because for each $x\in 2^\Gamma$ we have 
$$
	\{(\gamma_1,\gamma_2)\in \Omega : \,  H(\gamma_1,\gamma_2)(x) = 1 \} = 
	\{\gamma\in \Gamma: \, x_\gamma=1\} \times \{\gamma\in \Gamma: \, x_\gamma=0\} \in \Sigma.
$$
Since $\Ba(C_p(2^\Gamma,2)) = \mathcal{P}(C(2^\Gamma,2))$, we have $H^{-1}(X) \in \Sigma$ for every $X \subset C_p(2^\Gamma,2)$.
Thus, bearing in mind that $H$ is one-to-one, we conclude
that $\Sigma=\mathcal{P}(\Omega)$.
An appeal to Lemma~\ref{lem:MoreNoKunen}(i) ensures that $|\Gamma|$ is a Kunen cardinal. The proof is over.
\end{proof}

Recall that a compact space $K$ is called {\em dyadic} if $K$ is a continuous image of
$2^\kappa$ for some cardinal~$\kappa$; in this case, $\kappa$ can be taken to be equal 
to the weight of~$K$, see \cite[3.12.12]{eng}. The class of dyadic compacta of (infinite) weight $\kappa$
contains in particular $\kappa$-fold products of compact metrizable spaces.

\begin{cor}\label{dyadic}
If $K$ is a dyadic space and its weight is a Kunen cardinal, then $\Ba(C_p(K)) = \Bo(C(K))$.
\end{cor}
\begin{proof}
Let $\kappa$ be the weight of~$K$. 
If $\vf:2^\kappa\to K$ is a continuous surjection then the mapping
$T: C(K)\to C(2^\kappa)$, $T(g):=g\circ\vf$,  
is an isometric embedding which is pointwise continuous, so 
the assertion follows directly from Theorem \ref{thm:Main}.
\end{proof}

\begin{cor}\label{ell1sums}
Let $\{X_\alpha: \alpha < \kappa\}$ be a family of separable Banach spaces, where $\kappa$ is a Kunen cardinal. Then 
$X:=\bigoplus_{\ell^1} \{X_\alpha:\alpha < \kappa\}$ satisfies $\Ba(X_w) = \Bo(X)$.
\end{cor}
\begin{proof} If $\kappa$ is finite then $X$ is separable and so $\Ba(X_w) = \Bo(X)$. Suppose
$\kappa$ is infinite. Since each $(B_{X_\alpha^\ast},w^*)$ is a metrizable compact, 
there is a continuous surjection $2^\N\to B_{X_\alpha^\ast}$. Hence there is a
continuous surjection 
$$
	2^\kappa \to \prod_{\alpha<\kappa}B_{X_\alpha^\ast} = B_{X^\ast},
$$ 
so $X$ is isometric to a closed subspace of~$C(2^\kappa)$. Since $\Ba(C(2^\kappa)_w)=\Bo(C(2^\kappa))$
(by Theorem~\ref{thm:Main}), we have $\Ba(X_w) = \Bo(X)$ as well.
\end{proof}

\begin{cor}[Fremlin]\label{cor:Fremlin}
$\Ba(\ell^1(\Gamma)_w)=\Bo(\ell^1(\Gamma))$ if $|\Gamma|$ is a Kunen cardinal. 
\end{cor}

\begin{remark}\label{remark:dualCK}
Let $K$ be a compact space. 
\begin{enumerate}
\item[(i)] Suppose there exists a maximal family $\{\mu_\alpha:\alpha<\kappa\}$ of mutually singular Radon probabilities on~$K$
such that:
\begin{itemize}
\item $\kappa$ is a Kunen cardinal, 
\item each $L^1(\mu_\alpha)$ is separable. 
\end{itemize}
Then $\Ba(C(K)_w^*)=\Bo(C(K)^*)$, because 
$C(K)^*$ is isomorphic to the space $\bigoplus_{\ell^1} \{L^1(\mu_\alpha):\alpha < \kappa\}$ (cf. 
\cite[proof of Proposition~4.3.8]{alb-kal}).

\item[(ii)] The existence of a family $\{\mu_\alpha:\alpha<\kappa\}$ as in~(i) is guaranteed if:
\begin{itemize}
\item $|K|=\mathfrak{c}$ is Kunen, 
\item ${\rm span}\{\delta_t:t\in K\}$ is sequentially $w^*$-dense in~$C(K)^*$,
\item $L^1(\mu)$ is separable for every Radon probability $\mu$ on~$K$.
\end{itemize}
Thus, assuming that $\mathfrak{c}$ is Kunen, the equality $\Ba(C(K)_w^*)=\Bo(C(K)^*)$ holds true whenever
$|K|=\mathfrak{c}$ and $K$ belongs to one of the following classes of compacta: Eberlein, Corson (under MA + non CH), Rosenthal,
linearly ordered, Radon-Nikod\'{y}m, etc. (see e.g.
\cite{dza-kun,mer-J} and the references therein).
\end{enumerate}
\end{remark}

\subsection{Proof of Lemma~\ref{Mainlemma}}\label{subsection:lemma}

This subsection is devoted to prove Lemma~\ref{Mainlemma} above. The proof
is divided into several auxiliary lemmas. Throughout, 
$\tau$ is an $n$-type, $(Y,\Sigma)$ is a measurable space and $\Phi:\Gamma^n \to \mathcal{P}(Y)$ 
is a multifunction satisfying: 
\begin{itemize}
\item[(S)] For each $U\subset\Gamma$ and each closed set $I\subset \mathbb{R}$, the sets
$$
	\Phi\bigl(\{\gamma\in\Gamma^n : \, \tau(1_U(\gamma_1),\dots,1_U(\gamma_n))\subset I\}\bigr)
$$
$$
	\Phi\bigl(\{\gamma\in\Gamma^n : \, \tau(1_U(\gamma_1),\dots,1_U(\gamma_n))\cap I = \emptyset\}\bigr)
$$
are $\Sigma$-separated. 
\end{itemize}

\begin{defi}
Let $E$ be an equivalence relation on $\{1,\dots,n\}\times\{0,1\}$. We say that $E$ is a {\em $\tau$-proximality relation} (and we write $E\in Prox(\tau)$) 
if $\tau(\gamma^0)\cap \tau(\gamma^1) \neq \emptyset$ whenever $\gamma^0,\gamma^1\in 2^n$ satisfy
$$
	(p,i) E (q,j) \ \impli \ \gamma^i_p = \gamma^j_q
$$
for every $(p,i),(q,j)\in \{1,\dots,n\}\times\{0,1\}$.
\end{defi}

\begin{lem}\label{lem:Prox}
Let $\gamma^0,\gamma^1\in\Gamma^n$. 
The following statements are equivalent:
\begin{enumerate}
\item[(i)] $\gamma^0,\gamma^1$ are $\tau$-proximal.
\item[(ii)] There is $E\in Prox(\tau)$ such that 
$$
	(p,i) E (q,j) \ \impli \ \gamma^i_p = \gamma^j_q
$$
for every $(p,i),(q,j)\in \{1,\dots,n\}\times\{0,1\}$.
\end{enumerate}
\end{lem}
\begin{proof}
(i)$\impli$(ii). The equivalence relation~$E$ on $\{1,\dots,n\}\times \{0,1\}$ defined by 
$$
	(p,i) E (q,j) \ \Leftrightarrow  \ \gamma^i_p = \gamma^j_q
$$
is a $\tau$-proximality relation. Indeed, let $\delta^0,\delta^1\in 2^n$ satisfy the condition: 
$$
	(p,i) E (q,j) \ \impli \ \delta^i_p = \delta^j_q
$$
for every $(p,i),(q,j)\in \{1,\dots,n\}\times\{0,1\}$. Let $U \subset \Gamma$
be the set made up of all $\gamma^0_p$'s with $\delta^0_p=1$ and all
$\gamma^1_p$'s with $\delta^1_p=1$. Then
$$
	\tau(1_U(\gamma^i_1),\dots,1_U(\gamma^i_n))=\tau(\delta^i)
	\quad
	\mbox{for }i\in \{0,1\}
$$
and so the $\tau$-proximality of $\gamma^0$ and~$\gamma^1$ implies that $\tau(\delta^0)\cap \tau(\delta^1)\neq \emptyset$.

(ii)$\impli$(i). Fix $U \subset \Gamma$ and set
$$
	\delta^i:=(1_U(\gamma^i_1),\dots,1_U(\gamma^i_n)) \in 2^n 
	\quad 
	\mbox{for }i\in \{0,1\}.
$$
Observe that if $(p,i) E (q,j)$ then $\gamma^i_p=\gamma^j_q$ and so $1_U(\gamma^i_p)=1_U(\gamma^j_q)$.
Bearing in mind that $E\in Prox(\tau)$, we conclude that
$$
	\tau(1_U(\gamma^0_1),\dots,1_U(\gamma^0_n))
	\cap
	\tau(1_U(\gamma^1_1),\dots,1_U(\gamma^1_n))
	=
	\tau(\delta^0)
	\cap 
	\tau(\delta^1)
	\neq \emptyset.
$$
This shows that $\gamma^0$ and $\gamma^1$ are $\tau$-proximal.
\end{proof}

\begin{defi}
Let $E\in Prox(\tau)$. 
\begin{enumerate}
\item[(i)] An equivalence class $\mathcal{C}$ of $E$ is called a {\em linking class} if 
$\mathcal{C}=[(p,0)]=[(q,1)]$ for some $p,q\in \{1,\dots,n\}$. We denote by $\ell_E$ the set of linking equivalence classes of~$E$.

\item[(ii)] Let $i\in\{0,1\}$ and $A\subset \Gamma^n$. We define $L^i_E(A)$ as the set of 
all $\tilde{\gamma} \in \Gamma^{\ell_E}$ for which there is $\gamma\in A$ such that:
\begin{itemize}
\item $\gamma_{p} = \gamma_{q}$ whenever $(p,i) E (q,i)$;
\item $\gamma_k = \tilde{\gamma}_{[(k,i)]}$ whenever $[(k,i)]\in\ell_E$.
\end{itemize}

\end{enumerate}
\end{defi}

\begin{lem}\label{lem:tauseparation}
Let $A,B\subset \Gamma^n$. The following statements are equivalent:
\begin{enumerate}
\item[(i)] $A$ and $B$ are $\tau$-separated. 
\item[(ii)] $L^0_E(A) \cap L^1_E(B) = \emptyset$ for every $E\in Prox(\tau)$.
\end{enumerate}
\end{lem}
\begin{proof}
(i)$\impli$(ii).
Suppose that $L^0_E(A) \cap L^1_E(B) \neq \emptyset$ for some $E\in Prox(\tau)$. Take
$\tilde{\gamma} \in L^0_E(A)\cap L^1_E(B)$ and choose $\gamma^0\in A$, $\gamma^1\in B$, 
such that for $i\in \{0,1\}$ we have
\begin{center}
$\gamma^i_p=\gamma^i_q$ whenever $(p,i)E(q,i)$ and $\gamma^i_k=\tilde{\gamma}_{[(k,i)]}$ for every $[(k,i)]\in \ell_E$.
\end{center}
Therefore, 	$\gamma^i_p = \gamma^j_q$ whenever $(p,i) E (q,j)$. An appeal to
Lemma~\ref{lem:Prox} ensures that $\gamma^0$ and $\gamma^1$ are $\tau$-proximal, so 
$A$ and $B$ are not $\tau$-separated.

(ii)$\impli$(i).
If $A$ and $B$ are not $\tau$-separated, then (by Lemma~\ref{lem:Prox}) 
there exist $\gamma^0\in A$, $\gamma^1\in B$ and $E\in Prox(\tau)$ such that
$$
	(p,i) E (q,j) \ \impli \ \gamma^i_p = \gamma^j_q
$$
for every $(p,i),(q,j)\in \{1,\dots,n\}\times\{0,1\}$. Then we can define 
$\tilde{\gamma}\in \Gamma^{\ell_E}$
by saying that $\tilde{\gamma}_{[(p,i)]}:=\gamma^i_p$
for every $[(p,i)]\in \ell_E$. Clearly, 
$\tilde{\gamma} \in L^0_E(A)\cap L^1_E(B)$. 
\end{proof}

\begin{remark}\label{pieceseparation}
Let $U_n,V_n \subset Y$, $n\in \Nat$. If 
$U_n$ and $V_m$ are $\Sigma$-separated for every $n,m\in \Nat$, then 
$\bigcup_{n\in \Nat}U_n$ and $\bigcup_{n\in \Nat}V_n$ are $\Sigma$-separated as well.
\end{remark}
\begin{proof}
For each $n,m\in \Nat$, fix $S_{n,m}\in \Sigma$ such that $U_n\subset S_{n,m}$ and $V_m\cap S_{n,m}= \emptyset$. Then $S := \bigcup_{n\in \Nat}\bigcap_{m\in \Nat}S_{n,m}\in\Sigma$ satisfies 
$\bigcup_{n\in \Nat}U_n \subset S$ and $\left(\bigcup_{n\in \Nat}V_n\right) \cap S=\emptyset$.
\end{proof}

\begin{lem}\label{monotoneclassE}
Let $E_0\in Prox(\tau)$. For each $E\in Prox(\tau)\setminus\{E_0\}$, let us fix disjoint sets $X_E,Y_E\subset\Gamma^{\ell_E}$. 
Let $\mathfrak V$ be the family of all $W\subset\Gamma^{\ell_{E_0}}$ for which the following statement holds:
\begin{quote}
``If $A,B\subset \Gamma^n$ satisfy
\begin{itemize}
\item $L^0_E(A)\subset X_E$ and $L^1_E(B)\subset Y_E$ for every $E\in Prox(\tau)\setminus \{E_0\}$,
\item $L^0_{E_0}(A)\subset W$ and $L^1_{E_0}(B)\cap W = \emptyset$,
\end{itemize}
then $\Phi(A)$ and $\Phi(B)$ are $\Sigma$-separated.''
\end{quote}
Then $\mathfrak V$ is closed under countable unions and countable intersections.
\end{lem}

\begin{proof}
Let $(W_m)_{m\in\Nat}$ be an arbitrary sequence in~$\mathfrak{V}$.
We shall prove first that $W:=\bigcup_{m\in \Nat} W_m\in \mathfrak{V}$. For let $A,B\subset \Gamma^n$ be sets satisfying 
\begin{itemize}
\item[(i)] $L^0_E(A)\subset X_E$ and $L^1_E(B)\subset Y_E$ for every $E\in Prox(\tau)\setminus \{E_0\}$,
\item[(ii)] $L^0_{E_0}(A)\subset W$ and $L^1_{E_0}(B)\cap W = \emptyset$.
\end{itemize}
Note that for every $\gamma\in \Gamma^n$ the set $L^0_{E_0}(\{\gamma\})$ is either empty
or a singleton. For each $m\in \Nat$, define 
$$
	A_m := \{\gamma\in A : \, L^0_{E_0}(\{\gamma\})\subset W_m\}.
$$
Since $\bigcup_{\gamma \in A} L^0_{E_0}(\{\gamma\})=L^0_{E_0}(A) \subset W$, we have $A=\bigcup_{m\in \Nat} A_m$. 
Thus, bearing in mind Remark~\ref{pieceseparation}, in order to prove that $\Phi(A)=\bigcup_{m\in \Nat} \Phi(A_m)$ and $\Phi(B)$ are $\Sigma$-separated
it suffices to check that, for each $m\in \Nat$, the sets $\Phi(A_m)$ and $\Phi(B)$ are $\Sigma$-separated.
Fix $m\in \Nat$ and observe that:
\begin{itemize}
\item $L^0_E(A_m)\subset L^0_E(A) \subset X_E$ and $L^1_E(B)\subset Y_E$ for $E\in Prox(\tau)\setminus \{E_0\}$
(by~(i)),
\item $L^0_{E_0}(A_m)= \bigcup_{\gamma \in A_m} L^0_{E_0}(\{\gamma\}) \subset W_m$  
and $L^1_{E_0}(B)\cap W_m = \emptyset$ (by~(ii)).
\end{itemize}
Since $W_m\in\mathfrak V$ we conclude that $\Phi(A_m)$ and $\Phi(B)$ are $\Sigma$-separated, as desired. 
It follows that $W\in \mathfrak{V}$.

We now prove that $W':=\bigcap_{m\in \Nat} W_m \in \mathfrak{V}$. Fix $A,B\subset \Gamma^n$  
such that 
\begin{itemize}
\item[(i')] $L^0_E(A)\subset X_E$ and $L^1_E(B)\subset Y_E$ for every $E\in Prox(\tau)\setminus \{E_0\}$,
\item[(ii')] $L^0_{E_0}(A)\subset W'$ and $L^1_{E_0}(B)\cap W' = \emptyset$.
\end{itemize}
For each $m\in \Nat$ we define 
$$ 
	B_m := \{\gamma\in B : \, L^1_{E_0}(\{\gamma\})\cap W_m=\emptyset\}.
$$
Since each $L^1_{E_0}(\{\gamma\})$ is either empty or a singleton, and
$$
	\bigcup_{\gamma\in B}L^1_{E_0}(\{\gamma\})=L^1_{E_0}(B) \subset \Gamma^{\ell_{E_0}} \setminus W'=\bigcup_{m\in \Nat}
	\Gamma^{\ell_{E_0}} \setminus W_m,
$$ 
we have $B=\bigcup_{m\in \Nat} B_m$. Therefore, to show that $\Phi(A)$ and $\Phi(B)=\bigcup_{m\in \Nat} \Phi(B_m)$ 
are $\Sigma$-separated
it is enough to check that, for each $m\in \Nat$, the sets $\Phi(A)$ and $\Phi(B_m)$ are $\Sigma$-separated.
This follows immediately from the facts that $W_m\in \mathfrak{V}$ and
\begin{itemize}
\item $L^0_E(A)\subset X_E$ and $L^1_E(B_m)\subset L^1_E(B) \subset Y_E$ for $E\in Prox(\tau)\setminus\{E_0\}$ (by~(i')).
\item $L^0_{E_0}(A)\subset W'\subset W_m$ (by~(ii')) and 
$$
	L^1_{E_0}(B_m)=
	\bigcup_{\gamma \in B_m}L^1_{E_0}(\{\gamma\}) \subset \Gamma^{\ell_{E_0}}\setminus W_m.
$$
\end{itemize}
This proves that $W'\in \mathfrak{V}$ and we are done.
\end{proof}

\begin{defi}
Let $\Omega$ be a set and $A_1,\dots,A_m \in \mathcal{P}(\Omega)$. We say that $C \subset \Omega$
is an {\em atom} of the algebra on~$\Omega$ generated by $A_1,\dots,A_m$ if $C$ is nonempty and can be written
as $C=\bigcap_{i=1}^m D_i$ where each $D_i\in\{A_i,\Omega \setminus A_i\}$.  
\end{defi}

\begin{defi}
A set $W \subset \Gamma^n$ is called a {\em product} if it can be expressed as $W=\prod_{i=1}^n W_i$ for some
$W_i \subset \Gamma$ (which are called the {\em factors} of~$W$). 
\end{defi}

\begin{lem}\label{productseparation}
Let $A,B\subset \Gamma^n$ be products. If 
$A$ and $B$ are $\tau$-separated, then $\Phi(A)$ and $\Phi(B)$ are $\Sigma$-separated. 
\end{lem}
\begin{proof}
Write $A=\prod_{i=1}^n W_i$ and $B=\prod_{i=1}^n W'_i$.
Let $V_1,\dots,V_{m}$ 
be the atoms of the algebra on~$\Gamma$ generated by $W_1,\dots,W_n$ and $W'_1,\dots,W'_n$.
Then $A$ (resp. $B$) is the union of all products of the form $\prod_{i=1}^n V_{k_i}$
where $V_{k_i}\subset W_i$ (resp. $V_{k_i} \subset W'_i$). Thus, an appeal to
Remark~\ref{pieceseparation} allows us to assume that $A$ and~$B$ are of the form
$$
	A=\prod_{i=1}^n V_{k_i} \qquad
	B=\prod_{i=1}^n V_{r_i}	
$$
for some $k_i,r_i\in \{1,\dots,m\}$.

For each $j=1,\dots,m$ we choose $\gamma_j\in V_j$. Define $\gamma^0\in A$ and $\gamma^1\in B$ by
declaring $\gamma^0_i:=\gamma_{k_i}$ and $\gamma^1_i:=\gamma_{r_i}$ for $i\in\{1,\dots,n\}$.
Since $A$ and $B$ are $\tau$-separated, $\gamma^0$ and $\gamma^1$ 
are not $\tau$-proximal, so there exists $U\subset \Gamma$ such that
$$
	\tau(1_U(\gamma_{k_1}),\dots,1_U(\gamma_{k_n})) 
	\cap 
	\tau(1_U(\gamma_{r_1}),\dots,1_U(\gamma_{r_n})) = \emptyset.
$$
Define $V:= \bigcup\{V_j : \gamma_j\in U\} \subset \Gamma$. Observe that for each $i\in \{1,\dots,n\}$ we have
$\gamma_{k_i}\in U$ if and only if $\gamma_{k_i}\in V$, and
$\gamma_{r_i}\in U$ if and only if $\gamma_{r_i}\in V$. Therefore
\begin{equation}\label{equation:intersection}
	\tau(1_V(\gamma_{k_1}),\dots,1_V(\gamma_{k_n})) 
	\cap 
	\tau(1_V(\gamma_{r_1}),\dots,1_V(\gamma_{r_n})) = \emptyset.
\end{equation}
Set $I:=\tau(1_V(\gamma_{k_1}),\dots,1_V(\gamma_{k_n})) \subset \erre$. Observe that
for each $\delta \in A=\prod_{i=1}^n V_{k_i}$ and each $i\in\{1,\dots,n\}$, we have 
$\delta_i \in V$ if and only if $V_{k_i} \subset V$, which is equivalent to saying that 
$\gamma_{k_i}\in V$. In particular,
$$ 
	A  \subset \{\delta\in\Gamma^n : \, \tau(1_V(\delta_1),\dots,1_V(\delta_n)) \subset I\}. 
$$
In the same way, bearing in mind~\eqref{equation:intersection} we have
$$ 
	B  \subset \{\delta\in\Gamma^n : \, \tau(1_V(\delta_1),\dots,1_V(\delta_n))\cap I=\emptyset\}. 
$$
Now, property~(S) of~$\Phi$ implies that $\Phi(A)$ and $\Phi(B)$
are $\Sigma$-separated.
\end{proof}

Throughout the rest of the subsection we assume that
$|\Gamma|\leq \mathfrak{c}$, which is weaker than being a Kunen
cardinal (Lemma~\ref{lem:MoreNoKunen}(ii)). We can suppose without loss of generality
that $\Gamma \subset \erre$, so that $\Gamma^n$ is equipped with
the topology inherited from~$\erre^n$.

\begin{lem}\label{diagonalsets}
Let $A,B\subset \Gamma^n$ be open sets. If 
$A$ and $B$ are $\tau$-separated, then $\Phi(A)$ and $\Phi(B)$ are $\Sigma$-separated.
\end{lem}
\begin{proof} Let $\mathcal{O}_{A},\mathcal{O}_B \subset \erre^n$ be open sets such that
$A=\Gamma^n \cap \mathcal{O}_{A}$ and $B=\Gamma^n \cap \mathcal{O}_{B}$. Both $\mathcal{O}_{A},\mathcal{O}_B$ are countable unions of (open) 
products in~$\erre^n$ and, therefore, we can write
$A=\bigcup_{m\in \Nat}A_m$ and $B=\bigcup_{m\in \Nat}B_m$, where $A_m$ and $B_m$ 
are products in~$\Gamma^n$. For each $k,m\in \Nat$ 
the sets $A_k$ and $B_m$ are $\tau$-separated and Lemma~\ref{productseparation} ensures 
that $\Phi(A_k)$ and $\Phi(B_m)$ are $\Sigma$-separated.
Hence the sets $\Phi(A)=\bigcup_{m\in \Nat}\Phi(A_m)$ and $\Phi(B)=\bigcup_{m\in \Nat}\Phi(B_m)$
are $\Sigma$-separated (by Remark~\ref{pieceseparation}), as required.
\end{proof}

\begin{remark}\label{remark:AlgebraProducts}
The algebra on~$\Gamma^n$ generated by products is exactly the collection
of all subsets of~$\Gamma^n$ which can be written as a disjoint union of finitely many products.
\end{remark}
\begin{proof} Let us write $\mathcal{A}$ to denote such collection.
In order to prove that $\mathcal{A}$ is an algebra, observe first that
$\mathcal{A}$ is closed under finite intersections. On the other hand,
given any product $W=\prod_{i=1}^n W_i$, then $\Gamma^n \setminus W$ is the disjoint union
of all products of the form $\prod_{i=1}^n C_i$, where each $C_i$ is an atom of the algebra on~$\Gamma$
generated by $W_1,\dots,W_n$ and at least one $C_i$ is disjoint from~$W_i$. So, 
 $\Gamma^n \setminus W \in \mathcal{A}$. It follows that $\mathcal{A}$ is also closed under complements.
\end{proof}

\begin{lem}\label{lem:ParticularCase}
Let $A,B \subset \Gamma^n$ be such that for each $E\in Prox(\tau)$ there is $W_E\subset \Gamma^{\ell_E}$ in the algebra generated by products 
such that $L^0_E(A)\subset W_E$ and $L^1_E(B)\cap W_E = \emptyset$.
Then $\Phi(A)$ and $\Phi(B)$ are $\Sigma$-separated.
\end{lem}
\begin{proof} We divide the proof into several steps.

\smallskip
{\sc Step 1.} For each $E\in Prox(\tau)$, the set 
$W_E$ (resp. $\Gamma^{\ell_E} \setminus W_E$) is the union of a finite collection $\mathfrak{P}_E$ 
(resp.~$\mathfrak{Q}_E$) of 
products in~$\Gamma^{\ell_E}$ (Remark~\ref{remark:AlgebraProducts}). Observe also that $Prox(\tau)$ is finite.
Let $C_1,\dots,C_m$ be the atoms of the algebra on~$\Gamma$ generated by
the factors of all elements of the collection $\bigcup\{\mathfrak{P}_E \cup \mathfrak{Q}_E:E\in Prox(\tau)\}$. Then 
each $W_E$ (resp. $\Gamma^{\ell_E} \setminus W_E$) is a finite union of products with
factors in~$\{C_1,\dots,C_m\}$.

We can suppose without loss of generality that
$C_k \subset I_k:=(2k,2k+1) \subset \erre$ for all $k\in\{1,\dots,m\}$. Thus, if $A\subset \Gamma^n$
is any product with factors in~$\{C_1,\dots,C_m\}\cup\{\Gamma\}$, then 
 $A$ is open in~$\Gamma^n$, because it can be written as 
$A=\Gamma^n \cap P$ for some product $P\subset \erre^n$ with factors
in $\{I_1,\dots,I_m\}\cup\{\erre\}$.

\smallskip
{\sc Step 2.} Fix $E\in Prox(\tau)$. For $i\in\{0,1\}$, consider the equivalence relation $\approx_E^i$ on $\{1,\dots,n\}$ given by
$$
	p \approx_E^i q \ \Leftrightarrow \ (p,i)E(q,i).
$$
Set 
$$
	D_{\approx_E^i} := \{\gamma \in \Gamma^n : \, p\approx_E^i q \Rightarrow \gamma_p = \gamma_q\}
$$
and define $\varphi_E^i:D_{\approx_E^i} \to \Gamma^{\ell_E}$ by  
$$
	\varphi_E^i(\gamma)_{[(k,i)]}:=\gamma_k, \quad [(k,i)]\in \ell_E, \quad \gamma \in D_{\approx_E^i}.
$$

Let $R\subset \Gamma^{\ell_E}$ be any product with factors in
$\{C_1,\dots,C_m\}$. It is easy to check that
there is some product $A \subset \Gamma^n$ with factors in
$\{C_1,\dots,C_m\}\cup\{\Gamma\}$ (in particular, $A$~is open in~$\Gamma^n$) such that
$(\varphi_E^i)^{-1}(R)=D_{\approx_E^i}\cap A$, hence
$$
	(\varphi_E^i)^{-1}(R)\cup \Gamma^n \setminus D_{\approx_E^i}=
	A\cup \Gamma^n \setminus D_{\approx_E^i}.
$$
Since 
$$
	\Gamma^n \setminus D_{\approx_E^i}=
	\Gamma^n 
	\cap 
	\bigcup_{p \approx_E^i q}
	\{\gamma\in \erre^n: \, \gamma_p\neq \gamma_q\},
$$
we conclude that $(\varphi_E^i)^{-1}(R)\cup \Gamma^n \setminus D_{\approx_E^i}$ is
open in~$\Gamma^n$. 

It follows that the sets
$$
	\tilde{A}_E:=(\varphi_E^0)^{-1}(W_E)\cup \Gamma^n \setminus D_{\approx_E^0}
$$
$$
	\tilde{B}_E:=(\varphi_E^1)^{-1}(\Gamma^{\ell_E}\setminus W_E)\cup \Gamma^n \setminus D_{\approx_E^1}
$$
are open in~$\Gamma^n$. Moreover, since
$$
	L^0_E(S)=\varphi_E^0\bigl(S\cap D_{\approx_E^0}\bigr)
	\quad
	\mbox{and}
	\quad
	L^1_E(S)=\varphi_E^1\bigl(S\cap D_{\approx_E^1}\bigr)
	\quad
	\mbox{for every }S\subset \Gamma^n,
$$
we have:
\begin{itemize}
\item $\varphi_E^0(\gamma)\in L^0_E(A) \subset W_E$ for every $\gamma \in A \cap D_{\approx_E^0}$, hence $A \subset \tilde{A}_E$; 
\item $L_E^0(\tilde{A}_E)= \varphi_E^0(\tilde{A}_E\cap D_{\approx_E^0}) \subset W_E$;
\item $\varphi_E^1(\gamma)\in L^1_E(B) \subset \Gamma^{\ell_E}\setminus W_E$ for every $\gamma \in B \cap D_{\approx_E^1}$, hence
$B \subset \tilde{B}_E$;
\item $L_E^1(\tilde{B}_E)= \varphi_E^1(\tilde{B}_E\cap D_{\approx_E^1}) \subset \Gamma^{\ell_E}\setminus W_E$.
\end{itemize}

\smallskip
{\sc Step~3.} Now let 
$$
	\tilde{A} := \bigcap_{E\in Prox(\tau)} \tilde{A}_E
	\qquad
	\mbox{and}
	\qquad
	\tilde{B} := \bigcap_{E\in Prox(\tau)} \tilde{B}_E.
$$ 
For each $E\in Prox(\tau)$ we have 
$$
	L^0_E(\tilde{A})\cap L_E^1(\tilde{B}) \subset L_E^0(\tilde{A}_E) \cap  L_E^1(\tilde{B}_E)
	\subset W_E \cap (\Gamma^{\ell_E}\setminus W_E)=\emptyset,
$$
hence Lemma~\ref{lem:tauseparation} ensures that $\tilde{A}$ and $\tilde{B}$ are $\tau$-separated.
Since $\tilde{A}$ and $\tilde{B}$ are open in~$\Gamma^n$
(bear in mind that $Prox(\tau)$ is finite), an appeal 
to Lemma~\ref{diagonalsets} allows us to deduce
that $\Phi(\tilde{A})$ and $\Phi(\tilde{B})$ are $\Sigma$-separated. 
But $A\subset \tilde{A}$ and $B\subset\tilde{B}$, so the sets 
$\Phi(A)$ and $\Phi(B)$ are $\Sigma$-separated as well.
This finishes the proof.
\end{proof}

\begin{proof}[Proof of Lemma~\ref{Mainlemma}]
In view of Lemma~\ref{lem:tauseparation}, it suffices to prove that,
for any set $\mathcal{R}\subset Prox(\tau)$, the following statement holds:

\smallskip
\noindent $\langle\mathcal{R}\rangle$ {\em If $A,B \subset \Gamma^n$ satisfy:
\begin{itemize}
\item[(i)] $L^0_E(A)\cap L^1_E(B)=\emptyset$ for every $E\in \mathcal{R}$,
\item[(ii)] for each $E\in Prox(\tau) \setminus \mathcal{R}$ there is $W_E\subset \Gamma^{\ell_E}$ in the algebra generated by products 
such that $L^0_E(A)\subset W_E$ and $L^1_E(B)\cap W_E = \emptyset$,
\end{itemize}
then $\Phi(A)$ and $\Phi(B)$ are $\Sigma$-separated.}

We proceed by induction on~$|\mathcal{R}|$. The case $|\mathcal{R}|=0$ (i.e. $R=\emptyset$) has been proved in Lemma~\ref{lem:ParticularCase}.
So assume that $|\mathcal{R}|\geq 1$ and that $\langle \mathcal{R}' \rangle$ holds true for every 
subset of $Prox(\tau)$ with cardinality less than~$|\mathcal{R}|$. 
Take $A,B \subset \Gamma^n$ satisfying conditions (i) and~(ii) above. We will check that
$\Phi(A)$ and $\Phi(B)$ are $\Sigma$-separated.

Fix $E_0\in \mathcal{R}$
and set $\mathcal{R}':=\mathcal{R}\setminus \{E_0\}$. For each $E\in Prox(\tau) \setminus \{E_0\}$, fix
disjoint sets $X_E,Y_E \subset \Gamma^{\ell_E}$ as follows:
\begin{itemize}
\item $X_E:=L^0_E(A)$ and $Y_E:=L^1_E(B)$ for $E\in \mathcal{R}$,
\item $X_E:=W_E$ and $Y_E:=\Gamma^{\ell_E}\setminus W_E$ for $E\in Prox(\tau) \setminus \mathcal{R}$. 
\end{itemize} 

Let $\mathfrak{V}$ be as in Lemma~\ref{monotoneclassE}. {\em We claim that
every $W \subset \Gamma^{\ell_{E_0}}$ in the algebra generated by products
belongs to~$\mathfrak{V}$.} Indeed, let $A',B'\subset \Gamma^n$ be sets satisfying
$L^0_E(A')\subset X_E$ and $L^1_E(B')\subset Y_E$ for every $E\in Prox(\tau)\setminus \{E_0\}$,
$L^0_{E_0}(A')\subset W$ and $L^1_{E_0}(B')\cap W = \emptyset$. Then:
\begin{itemize}
\item $L^0_E(A')\cap L^1_E(B') \subset X_E \cap Y_E =\emptyset$ for every $E\in \mathcal{R}'$,
\item for each $E\in Prox(\tau) \setminus \mathcal{R}'$ 
there is $W'_E\subset \Gamma^{\ell_E}$ in the algebra generated by products 
such that $L^0_E(A')\subset W'_E$ and $L^1_E(B')\cap W'_E = \emptyset$ 
(take $W'_{E_0}:=W$ and $W'_E:=W_E$ for $E\neq E_0$). 
\end{itemize}
Since $\langle \mathcal{R}' \rangle$ holds, the sets 
$\Phi(A')$ and $\Phi(B')$ are $\Sigma$-separated. Therefore, $W\in \mathfrak{V}$.

Thus, $\mathfrak{V}$ contains the algebra on~$\Gamma^{\ell_{E_0}}$
generated by products. Since $\mathfrak{V}$ is a monotone class
(by Lemma~\ref{monotoneclassE}), from the Monotone Class Theorem it follows that the {\em $\sigma$-algebra}
on~$\Gamma^{\ell_{E_0}}$ generated by products is contained in~$\mathfrak{V}$. Now, the fact that
$|\Gamma|$ is a Kunen cardinal implies that $\mathfrak{V}=\mathcal{P}(\Gamma^{\ell_{E_0}})$.

In particular, the set $W:=L^0_{E_0}(A)$ belongs to~$\mathfrak{V}$.  
Since $L^0_E(A)\subset X_E$ and $L^1_E(B)\subset Y_E$ for every $E\in Prox(\tau)\setminus \{E_0\}$,
$L^0_{E_0}(A)\subset W$ and $L^1_{E_0}(B)\cap W = \emptyset$, we conclude that
$\Phi(A)$ and $\Phi(B)$ are $\Sigma$-separated. This proves that
$\langle\mathcal{R}\rangle$ holds and the proof 
of Lemma~\ref{Mainlemma} is over.
\end{proof}

\section{The case of $C(2^{\omega_1})$}\label{section:omega1}

The aim of this section is to give a different, more direct proof of the equality $\Ba(C_p(2^{\omega_1}))=\Bo(C(2^{\omega_1}))$, 
see Theorem~\ref{2omega1} below.

We denote by $\mathfrak G$ the family of all open intervals of~$\mathbb{R}$ with rational endpoints and we write
$\JJ:=\bigcup_{n\in \Nat} \mathfrak{G}^n$. Given a compact space~$K$, $n\in \Nat$, $A\subset K^n$ and $J=(J_1,\dots,J_n) \in \mathfrak{G}^n$, we define
\begin{multline*}
	u(A,J):=\{g\in C(K): \, \mbox{there is } (x_1,\dots,x_n)\in A \\
	\mbox{such that } g(x_k)\in J_k \mbox{ for all }k=1,\dots,n\}.
\end{multline*}

\begin{remark}\label{remark:Closure}
In the previous conditions, we have $u(A,J)=u(\ol{A},J)$.
\end{remark}
\begin{proof}
For any $g\in C(K)$, the set $U:=\prod_{k=1}^n g^{-1}(J_k)\subset K^n$ is open, and therefore
$U\cap \ol{A} \neq \emptyset$ if and only if $U\cap A\neq \emptyset$.
\end{proof}

In Corollary~\ref{1:2} we shall isolate a property of a compact space~$K$ guaranteeing that $\Ba(C_p(K))=\Bo(C_p(K))$.
To this end we need a couple of lemmas.

\begin{lemma}\label{1:1}
Let $K$ be a compact space such that $u(F,J)\in \Ba(C_p(K))$ for every  
closed set $F \subset K^n$, every $J\in\mathfrak{G}^n$ and every $n\in \Nat$. 
Then $\Ba(C_p(K))=\Bo(C_p(K))$.
\end{lemma}

\begin{proof}
Let $G\subset C(K)$ be open for the pointwise convergence topology. 
For $n\in \Nat$ and $J=(J_1,\ldots,J_n)\in\mathfrak{G}^n$, set
$A_J:=\bigcup\{A \subset K^n: u(A,J)\subset G\}$, so that 
$u(A_J,J) \subset G$. We claim that
\begin{equation}\label{equation:us}
	G=\bigcup_{J \in \JJ} u(A_J,J).
\end{equation}
Indeed, given any $g\in G$, we can find $\{t_1,\ldots, t_n\}\subset K$ and $J=(J_1,\ldots,J_n)\in\mathfrak{G}^n$ such that 
$$
	g\in H:=\{h\in C(K): \, h(t_k)\in J_k \mbox{ for all }k=1,\dots,n\}\subset G.
$$
Since $u(\{(t_1,\dots,t_n)\},J)=H \subset G$, we have $(t_1,\dots,t_n)\in A_J$ and
so $g\in u(A_J,J)$. This proves equality~\eqref{equation:us}.
Now, in view of Remark~\ref{remark:Closure}, we get
$$
	G=\bigcup_{J \in \JJ} u(\overline{A_J},J).
$$
Since $\JJ$ is countable and each $u(\overline{A_J},J)$ belongs to~$\Ba(C_p(K))$ (by the assumption), 
it follows that $G\in \Ba(C_p(K))$. Hence $\Ba(C_p(K))=\Bo(C_p(K))$.
\end{proof}

\begin{lem}\label{lem:1:medio}
Let $K$ be a compact space, $n\in \Nat$, $J\in \mathfrak{G}^n$ and $(F_p)_{p\in \Nat}$ 
a decreasing sequence  of closed separable subsets of~$K^n$. Then 
$u(\bigcap_{p\in \Nat} F_p,J)\in \Ba(C_p(K))$.
\end{lem}
\begin{proof} We divide the proof into two steps.

\smallskip
{\sc Step~1.} $u(S,J)\in \Ba(C_p(K))$ for every closed separable set~$S \subset K^n$.
Indeed, take $D\subset S$ countable with $\overline{D}=S$. By Remark~\ref{remark:Closure}, we have
$$
	u(S,J)=u(D,J)=\bigcup_{x\in D}u(\{x\},J).
$$ 
Since each $u(\{x\},J)$ belongs to~$\Ba(C_p(K))$, the same holds for $u(S,J)$.

\smallskip
{\sc Step~2.} Write $J=(J_1,\dots,J_n)$ and set $F:=\bigcap_{p\in \Nat}F_p$.
For each $m\in \Nat$, choose 
$$
	J^m=(J^m_1,\ldots, J^m_n)\in\mathfrak{G}^n
$$ 
such that $\ol{J^m_k}\subset J^{m+1}_k$ and
$\bigcup_{m\in \Nat} J^m_k=J_k$ for every $m\in \Nat$ and $k\in\{1,\dots, n\}$. 
According to Step~1, in order to prove that $u(F,J)\in \Ba(C_p(K))$ it suffices
to check that
\begin{equation}\label{equation:G}
	u(F,J)=\bigcup_{m \in \Nat}\bigcap_{p\in \Nat} u(F_p,J^m).
\end{equation}
To this end, observe first that if $g\in u(F,J)$ then there is $(x_1,\dots,x_n)\in F$ such that
$g(x_k)\in J_k$ for all~$k$. Since $J_k=\bigcup_{m\in \Nat} J^m_k$ and
$J^m_k\subset J^{m+1}_k$, we can find $m\in \Nat$ large enough such that
$g(x_k)\in J^m_k$ for all~$k$, hence $g\in u(F,J^m) \subset \bigcap_{p\in \Nat}u(F_p,J^m)$. 

To check ``$\supset$'' in~\eqref{equation:G}, fix $g\in \bigcup_{m \in \Nat}\bigcap_{p\in \Nat} u(F_p,J^m)$.
Then there exists $m\in \Nat$ such that, for each $p\in \Nat$, there is some
$x^p=(x^p_1,\ldots, x^p_n)\in F_p$ with the property
that $g(x^p_k)\in J^m_k$ for all~$k$. Let $x\in K^n$ be
any cluster point of the sequence $(x^p)_{p\in \Nat}$. Then 
$x\in F$ and $g(x_k)\in \ol{J^m_k}\subset J_k$ for all~$k$, witnessing that
$g\in u(F,J)$. This proves~\eqref{equation:G} and we are done.
\end{proof}

As an immediate consequence of Lemmas~\ref{1:1} and~\ref{lem:1:medio} we get:

\begin{cor}\label{1:2}
Let $K$ be a compact space such that, for each $n\in \Nat$ and each closed set $F\subset K^n$, there is a decreasing
sequence $(F_p)_{p\in \Nat}$ of closed separable subsets of~$K^n$ such that 
$F=\bigcap_{p\in \Nat} F_p$.
Then $\Ba(C_p(K))=\Bo(C_p(K))$. 
\end{cor}

It turns out that the previous criterion can be applied to~$\too$, as we next show.

\begin{lemma}\label{2:1}
For each closed set $F\subset 2^{\omega_1}$ there is a decreasing
sequence $(F_p)_{p\in \Nat}$ of closed separable subsets of~$\too$ such that 
$F=\bigcap_{p\in \Nat} F_p$.
\end{lemma}

\begin{proof}
By Parovicenko's theorem (cf. \cite[3.12.18]{eng}), every compact space of weight less than or equal to~$\omega_1$
(like~$F$) is a continuous image of~$\beta\Nat\sm\Nat$.
Let $q:\beta\Nat\sm\Nat \to 2^{\omega_1}$ be a continuous mapping 
with $q(\beta\Nat\sm\Nat)=F$. Then $q$ can be extended to a continuous mapping $g:\beta\Nat \to 2^{\omega_1}$. Indeed, fix 
$\alpha<\omega_1$, let $\pi_\alpha: 2^{\omega_1} \to \{0,1\}$ be the $\alpha$-th coordinate projection and 
apply Tietze's theorem to find a continuous mapping $f_\alpha:\beta\Nat \to [0,1]$ such that
$f_\alpha|_{\beta\Nat\sm\Nat}=\pi_\alpha \circ q$.  Since $f_\alpha^{-1}(\{0\})$ and 
$f_\alpha^{-1}(\{1\})$ are disjoint closed subsets of the {\em $0$-dimensional} compact space~$\beta\Nat$, there is
a clopen set $A_\alpha\subset \beta\Nat$ such that $f_\alpha^{-1}(\{0\})\cap A_\alpha=\emptyset$
and $f_\alpha^{-1}(\{1\})\subset A_\alpha$. Now, it is easy to check 
that the continuous mapping $g:\beta\Nat \to 2^{\omega_1}$ defined by
$\pi_\alpha \circ g:=1_{A_\alpha}$ for all $\alpha < \omega_1$ satisfies $g|_{\beta\Nat \sm \Nat}=q$.

For each $p\in\Nat$, the set $Z_p:=\beta\Nat\sm \{1,\dots,p\}$ is closed and separable, hence
the same holds for $F_p:=g(Z_p) \subset 2^{\omega_1}$. Since
$(Z_{p})_{p\in \Nat}$ is a decreasing sequence of compact sets and~$g$ is continuous, we have
$$
	\bigcap_{p\in \Nat}F_p=
	\bigcap_{p\in  \Nat}g(Z_p)=
	g\Bigl(\bigcap_{p\in \Nat} Z_p\Bigr)=
	g(\beta\Nat\sm \Nat)=q(\beta\Nat\sm \Nat)=F,
$$
and the proof is over.
\end{proof}

Finally, we can give an alternative proof of the following:

\begin{thm}\label{2omega1}
$\Ba(C_p(2^{\omega_1})) = \Bo(C(2^{\omega_1}))$.
\end{thm}
\begin{proof}
As we pointed out in the introduction, for any cardinal~$\kappa$ we always have
$$
	\Bo(C_p(2^\kappa))=\Bo(C(2^\kappa)).
$$ 
On the other hand, $\Ba(C_p(2^{\omega_1})) = \Bo(C_p(2^{\omega_1}))$, 
by Corollary~\ref{1:2} and Lemma~\ref{2:1} (bear in mind that
all finite powers of~$2^{\omega_1}$ are homeomorphic to~$2^{\omega_1}$).
\end{proof}

\begin{remark}\label{remark:Parovicenko}
Let us say that $\kappa$ is a {\em Parovicenko cardinal} if every compact space of 
weight less than or equal to~$\kappa$ is a continuous image of $\beta\Nat\sm \Nat$.
This is the only property of the cardinal~$\omega_1$ that we have used
in the proofs of Lemma~\ref{2:1} and Theorem~\ref{2omega1}, so we have indeed shown 
that: 
\begin{center}
{\em $\Ba(C_p(2^\kappa)) = \Bo(C(2^\kappa))$ whenever $\kappa$ is a Parovicenko cardinal}. 
\end{center}
Notice that van Douwen and Przymusi\'{n}ski \cite{dou-pri}
proved that, under Martin's axiom, all cardinals $<\con$ are Parovicenko cardinals.
We do not known whether the analogue of Lemma \ref{2:1} for~$2^\kappa$ 
is true if $\kappa$ is a Kunen cardinal. 
\end{remark}

Recall that a Banach space $X$ is {\em measure-compact} (in its weak topology) if and only if, 
for each probability measure $\mu$ on $\Ba(X_w)$, there is a separable subspace 
$X_0$ of~$X$ such that $\mu^\ast(X_0)=1$. Such a property has been considered in connection with Pettis integration,
see e.g. \cite{edg1,tal}. The following consequence
of Theorem~\ref{2omega1} was first proved in~\cite{ple5} by a completely different approach.

\begin{cor}\label{cor:Grzegorz}
$C(\too)$ is measure-compact.
\end{cor}
\begin{proof}
Let $\mu$ be a probability measure on $\Ba(C_w(\too))=\Bo(C(\too))$.
Since the metric space $C(\too)$ has density character~$\omega_1$ (which is not 
real-valued measurable), a classical result due to Marczewski and Sikorski (cf.~\cite[Theorem III]{whe2})
ensures that $\mu$ has a separable support.
\end{proof}

In Corollary \ref{cor:Grzegorz} one can replace $\omega_1$ by any $\kappa$ which is a Kunen cardinal, since
in such a case no cardinal $\kappa_1\le\kappa$ is real-valued measurable, see~\cite{Kunen68}.
However, for $\kappa>\omega_1$ the result of \cite{ple5} is more general: under the absence of weakly
inaccessible cardinals $C(2^{\kappa})$ is measure-compact for every $\kappa$.

Let us also mention another consequence of Theorem \ref{2omega1}; cf.~\cite{Plachky92}
for some results on Borel structures in nonseparable metric spaces. We refer to \cite{vanDouwen} for the definition of cardinal $\mathfrak{p}$.

\begin{cor}[$\mathfrak{p}>\omega_1$]
$\Bo(C(2^{\omega_1}))$ is countably generated.
\end{cor}
\begin{proof}
Let $A\subset 2^{\omega_1}$ be a countable dense set
and let $\Sigma$ be the $\sigma$-algebra on~$C(2^{\omega_1})$ generated by $\{\delta_a:a\in A\}$. 
Clearly, $\Sigma$ is countably generated.
It follows from $\mathfrak{p}>\omega_1$ that
every $x\in 2^{\omega_1}$ is a limit of a converging sequence from~$A$, 
see e.g. \cite[Theorem~6.2]{vanDouwen}. This implies
that $\delta_x$ is $\Sigma$-measurable for every $x\in 2^{\omega_1}$, and we get
$\Sigma=\Ba(C_p(2^{\omega_1})) = \Bo(C(2^{\omega_1}))$, which completes the proof.
\end{proof}

\section{Non weak Baire measurable norms}\label{section:Norms}

An equivalent norm on a Banach space~$X$ is
$\Ba(X_w)$-measurable (as a real-valued function defined on~$X$) if and only if 
its balls belong to~$\Ba(X_w)$. Clearly, this implies that all singletons belong 
to~$\Ba(X_w)$, which is equivalent to saying that the dual $X^*$ is $w^*$-separable, cf.
\cite[Theorem 1.5.3]{hei}. There are Banach spaces with $w^*$-separable dual which admit 
a non $\Ba(X_w)$-measurable equivalent norm, like $\ell^\infty$ and the Johnson-Lindenstrauss
spaces, see~\cite{rod9}. Obviously, if the equality $\Ba(X_w)=\Bo(X)$ holds, then 
all equivalent norms on~$X$ are $\Ba(X_w)$-measurable. The aim of this section is to 
show that the converse holds for $C(2^\kappa)$ and $\ell^1(\kappa)$, see Corollary~\ref{renormell1}. 

Recall that a function $f:\Omega \to X$ from a measurable space $(\Omega,\Sigma)$ to a Banach space~$X$
is called {\em scalarly measurable} if the composition $x^*\circ f$ is $\Sigma$-measurable for every $x^*\in X^*$, i.e.
$f$ is $\Sigma$-$\Ba(X_w)$-measurable. We shall also use the following notion introduced in~\cite{gra-alt}:

\begin{defi}
Let $X$ be a Banach space. A family
$\{(x_\alpha,x_\alpha^*):\alpha \in I\} \subset X \times X^*$ is called
a {\em bounded almost biorthogonal system (BABS) of type $\eta \in [0,1)$} if
\begin{enumerate}
\item[(i)] $\{x_\alpha:\alpha \in I\}$ and $\{x_\alpha^*:\alpha\in I\}$ are bounded,
\item[(ii)] $x_\alpha^*(x_\alpha)=1$ for every $\alpha \in I$,
\item[(iii)] $|x_\alpha^*(x_\beta)|\leq \eta$ whenever $\alpha \neq \beta$.
\end{enumerate}
\end{defi}

\begin{lem}\label{pro:NonMeasurableFromUBABS}
Let $X$ be a Banach space having a BABS $\{(x_\alpha,x_\alpha^*):\alpha \in I\}$ of type $\eta \in [0,1)$.
Suppose there is a measurable space $(\Omega,\Sigma)$ and a mapping $i: \Omega \to I$ such that:
\begin{itemize}
\item the function $f:\Omega \to X$ defined by $f(\theta):=x_{i(\theta)}$ is scalarly
measurable,
\item there is $A\subset I$ such that $i^{-1}(A)\not \in \Sigma$.
\end{itemize}
Then there is an equivalent norm on~$X$ which 
is not $\Ba(X_w)$-measurable.
\end{lem}
\begin{proof}
Fix an equivalent norm $\|\cdot\|$ on~$X$ and set $C:=\sup\{\|x_\alpha\|:\alpha\in I\}$. 
The formula 
$$
	\|x\|_0:=C^{-1}\max\left\{\|x\|,C\sup_{\alpha\in I}|x_\alpha^*(x)| \right\}
$$
defines an equivalent norm on~$X$ (bear in mind that $\{x_\alpha^*:\alpha\in I\}$ is bounded)
such that $\|x_\alpha\|_0=1$ for all $\alpha \in I$. 
Fix $1<u<v<\eta^{-1}$ (with the convention $0^{-1}=\infty$) and set 
$b(\alpha):=u$ if $\alpha \in A$, $b(\alpha):=v$ if $\alpha \in I \setminus A$.
The formula 
$$
	|x|:=\max\left\{\|x\|_0,\sup_{\alpha\in I}b(\alpha)|x_\alpha^*(x)| \right\}
$$
defines another equivalent norm on~$X$. 

We claim that $|\cdot|$ is not $\Ba(X_w)$-measurable. To prove this, it suffices to check
that the real-valued function $\theta \mapsto |f(\theta)|$ is not $\Sigma$-measurable
(bear in mind that $f$ is $\Sigma$-$\Ba(X_w)$-measurable). 
Fix $\theta \in \Omega$. For each $\alpha \in I$ with $\alpha \neq i(\theta)$ we have $|x_\alpha^*(f(\theta))|=
|x_\alpha^*(x_{i(\theta)})|\leq \eta$ and so
$$
	b(\alpha)|x_\alpha^*(f(\theta))| \leq b(\alpha)\eta < 1 = \|f(\theta)\|_0.
$$
On the other hand, $b(i(\theta))|x_{i(\theta)}^*(f(\theta))|=b(i(\theta))>1=\|f(\theta)\|_0$. It follows that
\begin{multline*}
	|f(\theta)|=\max\left\{\|f(\theta)\|_0,\sup_{\alpha\in I}b(\alpha)|x_\alpha^*(f(\theta))| \right\}= \\ =
	b(i(\theta))=u1_{i^{-1}(A)}(\theta)+v1_{\Omega \setminus i^{-1}(A)}(\theta)
\end{multline*}
for all $\theta \in \Omega$. Since $i^{-1}(A)\not\in \Sigma$, the function $\theta \mapsto |f(\theta)|$ is 
not $\Sigma$-measurable.
\end{proof}

\begin{lem}\label{LemmaBABSKunen}
Let $X$ be a Banach space having a bounded biorthogonal system 
$\{(x_\alpha,x_\alpha^*):\alpha \in I\}$. Let $U \subset I\times I$ be a set
such that:
\begin{itemize}
\item[(a)] $\alpha \neq \beta$ for every $(\alpha,\beta)\in U$,
\item[(b)] $(\beta,\alpha) \not\in U$ whenever $(\alpha,\beta)\in U$. 
\end{itemize}
Then:
\begin{enumerate}
\item[(i)] The family 
\begin{equation}\label{equation:BABS}
	\left\{\left(x_\alpha+x_\beta,\frac{x_\alpha^*+x_\beta^*}{2}\right): \ (\alpha,\beta)\in U\right\}
	\subset
	X \times X^*
\end{equation}
is a BABS of type $1/2$.
\item[(ii)] The function $f:U\to X$ given by $f(\alpha,\beta):=x_\alpha + x_\beta$
is scalarly measurable when $U$ is equipped with the trace of $\mathcal{P}(I)\otimes\mathcal{P}(I)$.
\end{enumerate}
\end{lem}
\begin{proof} To prove~(i), fix $(\alpha,\beta)$ and $(\alpha',\beta')$ in~$U$. Then 
$$
	d:=(x_\alpha^*+x_\beta^*)(x_{\alpha'}+x_{\beta'})=
	\delta_{\alpha,\alpha'}+\delta_{\alpha,\beta'}+\delta_{\beta,\alpha'}+\delta_{\beta,\beta'}
$$
and therefore:
\begin{itemize}
\item If $(\alpha,\beta)=(\alpha',\beta')$, then $\alpha \neq \beta'$ and $\alpha'\neq \beta$ (by~(a)),
hence $d=2$.
\item If $\alpha=\alpha'$ and $\beta\neq \beta'$, 
then $\alpha \neq \beta'$ and $\alpha'\neq \beta$ (by~(a)), hence $d=1$.
\item If $\alpha\neq \alpha'$ and $\beta= \beta'$, 
then $\alpha \neq \beta'$ and $\alpha'\neq \beta$ (by~(a)), hence $d=1$.
\item If $\alpha\neq \alpha'$ and $\beta\neq \beta'$, 
then $d\in\{0,1\}$, because in this case we have
$\alpha \neq \beta'$ whenever $\alpha'=\beta$ (by~(b)).
\end{itemize}
It follows that \eqref{equation:BABS} is a BABS of type $1/2$.

To prove (ii), fix $x^\ast\in X^\ast$.  
For each $r\in \erre$, the set
\begin{multline*}
 \{(\alpha,\beta)\in U : \, x^\ast f(\alpha,\beta) < r\} 
= \{(\alpha,\beta)\in U : \, x^*(x_\alpha) + x^*(x_\beta) < r\}=\\ 
= \bigcup_{\substack{p,q\in\mathbb{Q}\\ p+q<r}} \{(\alpha,\beta)\in U : \, x^*(x_\alpha)<p, \, x^*(x_\beta)<q\}=\\ 
= U \cap \bigcup_{\substack{p,q\in\mathbb{Q}\\ p+q<r}}\{\alpha \in I: \, x^*(x_\alpha)<p\}\times\{\beta\in I: \,  x^*(x_\beta)<q\}
\end{multline*}
belongs to the trace of $\mathcal{P}(I)\otimes\mathcal{P}(I)$ on~$U$.
So, $f$ is scalarly measurable. 
\end{proof}

We arrive at the key result of this section.

\begin{theo}\label{renormeBABS}
Let $X$ be a Banach space having a biorthogonal system of non Kunen cardinality. 
Then there exists an equivalent norm on~$X$ 
which is not $\Ba(X_w)$-measurable.
\end{theo}
\begin{proof} 
Let $\kappa$ be a non Kunen cardinal such that $X$ has a biorthogonal system
of cardinality~$\kappa$. Suppose first that $\kappa>\mathfrak{c}$. Then $|X| > \mathfrak{c}$ and so 
$X^*$ is not $w^*$-separable (bear in mind that any Banach space
having $w^*$-separable dual injects into~$\ell^\infty$). 
Thus, in this case {\em all} equivalent norms on~$X$ are not $\Ba(X_w)$-measurable.

Suppose now that $\kappa \leq \mathfrak{c}$. Fix a {\em bounded} biorthogonal system
$$
	\{(x_\alpha,x_\alpha^*): \, \alpha \in I\} \subset X \times X^*
$$ 
with $|I|=\kappa$ (cf. \cite[Theorem~4.15]{fab-alt-JJ}).
We can assume that $I \subset \erre$. Then 
$$
	U:=\{(\alpha,\beta)\in I\times I: \, \alpha>\beta\}
	\quad \mbox{and} \quad
	V:=\{(\alpha,\beta)\in I \times I:\, \alpha<\beta\}
$$ 
belong to
$\mathcal{P}(I)\otimes \mathcal{P}(I)$, because they can be written as
$$
	U=\bigcup_{\substack{p,q\in \mathbb{Q}\\ p>q}}I\cap (p,\infty)\times I\cap (-\infty,q)
	\quad \mbox{and} \quad
	V=\bigcup_{\substack{p,q\in \mathbb{Q}\\ p<q}}I\cap (-\infty,p)\times I\cap (q,\infty).
$$
Since $|I|$ is not a Kunen cardinal, there is a set $B \subset I\times I$ which does not
belong to~$\mathcal{P}(I)\otimes \mathcal{P}(I)$. As we noticed in the proof
of Lemma~\ref{lem:MoreNoKunen}, we have
$$
	B \setminus (U\cup V) \in \mathcal{P}(I)\otimes \mathcal{P}(I),
$$
therefore either $B \cap U \not\in \mathcal{P}(I)\otimes \mathcal{P}(I)$
or $B\cap V\not\in\mathcal{P}(I)\otimes \mathcal{P}(I)$. From now on we 
assume that $B \cap U \not\in \mathcal{P}(I)\otimes \mathcal{P}(I)$ (the other case is analogous).

Let $\Sigma_U$ be the trace $\sigma$-algebra of~$\mathcal{P}(I)\otimes \mathcal{P}(I)$ on~$U$.
Observe that $U$ satisfies conditions (a) and~(b) of Lemma~\ref{LemmaBABSKunen}, hence
the family
$$	
	\left\{\left(x_\alpha+x_\beta,\frac{x_\alpha^*+x_\beta^*}{2}\right): \ (\alpha,\beta)\in U\right\}
	\subset
	X \times X^*
$$
is a BABS of type $1/2$ and the function $f:U \to X$ given by $f(\alpha,\beta):=x_\alpha+x_\beta$ 
is scalarly measurable with respect to~$\Sigma_U$. 
Since $A:=B \cap U\not \in\Sigma_U$ (bear in mind
that $\Sigma_U \subset \mathcal{P}(I)\otimes \mathcal{P}(I)$), an appeal
to Lemma~\ref{pro:NonMeasurableFromUBABS} ensures the existence
of a non $\Ba(X_w)$-measurable equivalent norm on~$X$. The proof is over.
\end{proof}

Let $\kappa$ be a cardinal.
For each $\alpha < \kappa$, define $(e_\alpha,e^*_\alpha) \in \ell^1(\kappa)\times \ell^1(\kappa)^*$ by declaring
$e_\alpha(\beta):=\delta_{\alpha,\beta}$ for all $\beta< \kappa$
and $e^*_\alpha(f):=f(\alpha)$ for all $f\in \ell^1(\kappa)$. Then
$\{(e_\alpha,e^*_\alpha):\alpha< \kappa\}$ is a biorthogonal system. Moreover, since $\ell^1(\kappa)$ is isomorphic to
a closed subspace of~$C(2^{\kappa})$, the Hahn-Banach theorem ensures
that $C(2^{\kappa})$ also has a 
biorthogonal system of cardinality~$\kappa$. From Theorems~\ref{thm:Main} 
and~\ref{renormeBABS} we now get:

\begin{cor}\label{renormell1}
The following statements are equivalent for a cardinal~$\kappa$:
\begin{enumerate}
\item[(i)] $\kappa$ is a Kunen cardinal. 
\item[(ii)] All equivalent norms on $\ell^1(\kappa)$ 
are $\Ba(\ell^1(\kappa)_w)$-measurable.
\item[(iii)] All equivalent norms on $C(2^{\kappa})$ 
are $\Ba(C_w(2^{\kappa}))$-measurable.
\end{enumerate}
\end{cor}

It is clear that an equivalent norm on a Banach space~$X$ is
$\Ba(X_w)$-measurable whenever its closed dual unit ball is $w^*$-separable.
However, the converse is not true in general
(for an example with $X=\ell^\infty$, see~\cite{rod9}).
On the other hand, it was shown in~\cite{gra-alt} that the following properties are equivalent:
\begin{enumerate}
\item[(i)] All equivalent norms on~$X$ have $w^*$-separable closed dual unit ball.
\item[(ii)] There is no {\em uncountable} BABS on~$X$.
\end{enumerate}
Moreover, when $X$ is a dual space, (i) and~(ii)
are equivalent to the separability of~$X$, cf. \cite[Corollary~4.34]{fab-alt-JJ}. 
Our last result complements such equivalence.

\begin{pro}\label{normsinduals}
Let $Y$ be a separable Banach space not containing~$\ell^1$. The following statements are equivalent:
\begin{enumerate}
\item[(i)] $Y^*$ is separable.
\item[(ii)] All equivalent norms on~$Y^*$ are $\Ba(Y^*_w)$-measurable.
\end{enumerate}
\end{pro}
\begin{proof}
It only remains to prove (ii)$\impli$(i). Since $Y$ is separable, its dual $X:=Y^*$ is 
a representable Banach space. Thus, if we assume that $X$ is not separable, then 
there is a bounded biorthogonal system $\{(x_\alpha,x_\alpha^\ast) : \alpha<\mathfrak c\}
\subset X \times X^*$, cf. \cite[Theorem~4.33]{fab-alt-JJ}.
Let $D\subset Y$ be a countable 
norm dense set. We claim that 
\begin{equation}\label{equation:CountablyGenerated}
	\Ba(X_w)=\sigma(D).
\end{equation} 
Indeed, fix $y^{**}\in X^\ast=Y^{**}$. By the Odell-Rosenthal theorem
(cf. \cite[Theorem~4.1]{van}) there is a sequence $(y_n)_{n\in \Nat}$ in~$Y$
converging to~$y^{**}$ in the $w^*$-topology. Since $D$ is norm dense in~$Y$, we can find $y'_n \in D$ such that $\|y_n-y'_n\| \leq 1/n$. Then
$(y'_n)_{n\in \Nat}$ also converges to~$y^{**}$ in the $w^*$-topology and so
$y^{**}$ is $\sigma(D)$-measurable. As $y^{**}\in X^*$ is arbitrary, 
equality~\eqref{equation:CountablyGenerated} holds.

In particular, $\Ba(X_w)$ is countably generated. Thus, $|\Ba(X_w)| = \mathfrak c<2^\mathfrak{c}$
and hence there exists $A\subset\mathfrak c$ such that 
$\{x_\alpha: \alpha\in A\}$ does not belong to the trace
 of~$\Ba(X_w)$ on~$\Omega:=\{x_\alpha:\alpha<\mathfrak{c}\}$, which we denote 
by~$\Sigma$. Since the ``identity'' function $f:\Omega \to X$ 
satisfies the assumptions of~Lemma~\ref{pro:NonMeasurableFromUBABS}
(with respect to~$\Sigma$), the space $X$ admits a non $\Ba(X_w)$-measurable
equivalent norm. 
\end{proof}

\begin{remark}
If $\mathfrak{c}$ is not a Kunen cardinal, then statements (i) and~(ii) of Proposition~\ref{normsinduals} 
are equivalent for any separable Banach space~$Y$. 
\end{remark}
\begin{proof}
It only remains to prove that (ii) fails when $Y$ contains~$\ell^1$. In this case, 
$\ell^1(\mathfrak c)$ is isomorphic to a closed subspace $Z$ of~$Y^\ast$
(cf. \cite[Theorem~4.1]{van}). By Corollary~\ref{renormell1}, there is a non
$\Ba(Z_w)$-measurable equivalent norm $\|\cdot\|_Z$ on~$Z$.
Since the trace of $\Ba(Y^*_w)$ on~$Z$ is exactly $\Ba(Z_w)$, 
we conclude that any equivalent norm on~$Y^*$ extending~$\|\cdot\|_Z$ (cf. \cite[II.8.1]{dev-alt-J})
cannot be $\Ba(Y^*_w)$-measurable.
\end{proof}

However, if $\mathfrak{c}$ is a Kunen cardinal, then
$\Ba(C[0,1]^*_w)=\Bo(C[0,1]^*)$ (see Remark~\ref{remark:dualCK}) and
so all equivalent norms on~$C[0,1]^*$ are $\Ba(C[0,1]^*_w)$-measurable, while
$C[0,1]^*$ is nonseparable.


\begin{thebibliography}{10}

\bibitem{alb-kal}
F.~Albiac and N.~J. Kalton, \emph{Topics in {B}anach space theory}, Graduate
  Texts in Mathematics, vol. 233, Springer, New York, 2006.

\bibitem{bur-pol-1}
D.~K. Burke and R.~Pol, \emph{On {B}orel sets in function spaces with the weak
  topology}, J. London Math. Soc. \textbf{68} (2003), no.~2, 725--738.

\bibitem{bur-pol-2}
D.~K. Burke and R.~Pol, \emph{Non-measurability of evaluation maps on subsequentially complete
  {B}oolean algebras}, New Zealand J. Math. \textbf{37} (2008), no.~2, 9--13.

\bibitem{dev-alt-J}
R.~Deville, G.~Godefroy, and V.~Zizler, \emph{Smoothness and renormings in
  {B}anach spaces}, Pitman Monographs and Surveys in Pure and Applied
  Mathematics, vol.~64, Longman Scientific \& Technical, Harlow, 1993.

\bibitem{vanDouwen}  E.~K. van Douwen, \emph{The integers and topology}, Handbook of set-theoretic topology, 111--167, North-Holland, Amsterdam, 1984.

\bibitem{dou-pri}
E.~K. van Douwen and T.~C. Przymusinski, \emph{Separable extensions of first
  countable spaces}, Fund. Math. \textbf{105} (1979), 147--158.

\bibitem{van}
D.~van Dulst, \emph{Characterizations of {B}anach spaces not containing {$l{\sp
  1}$}}, CWI Tract, vol.~59, Centrum voor Wiskunde en Informatica, Amsterdam,
  1989. 

\bibitem{dza-kun}
M.~D{\u{z}}amonja and K.~Kunen, \emph{Properties of the class of measure
  separable compact spaces}, Fund. Math. \textbf{147} (1995), no.~3, 261--277.
  
\bibitem{edg-J}
G.~A. Edgar, \emph{Measurability in a {B}anach space}, Indiana Univ. Math. J.
  \textbf{26} (1977), no.~4, 663--677.

\bibitem{edg1}
G.~A. Edgar, \emph{Measurability in a {B}anach space. {II}}, Indiana Univ. Math. J.
  \textbf{28} (1979), no.~4, 559--579. 

\bibitem{eng}
R.~Engelking, \emph{General topology}, PWN---Polish Scientific Publishers,
  Warsaw, 1977, Translated from the Polish by the author, Monografie
  Matematyczne, Tom 60. [Mathematical Monographs, Vol. 60].

\bibitem{fre6}
D.~H. Fremlin, \emph{Borel sets in nonseparable {B}anach spaces}, Hokkaido
  Math. J. \textbf{9} (1980), no.~2, 179--183. 

\bibitem{freMT-4}
D.~H. Fremlin, \emph{Measure theory. {V}ol. 4}, Torres Fremlin, Colchester, 2006,
  Topological measure spaces. Part I, II, Corrected second printing of the 2003
  original.

\bibitem{gra-alt}
A.~S. Granero, M.~Jim{\'e}nez, A.~Montesinos, J.~P. Moreno, and A.~Plichko,
  \emph{On the {K}unen-{S}helah properties in {B}anach spaces}, Studia Math.
  \textbf{157} (2003), no.~2, 97--120.

\bibitem{fab-alt-JJ}
P.~H{\'a}jek, V.~Montesinos~Santaluc{\'{\i}}a, J.~Vanderwerff, and V.~Zizler,
  \emph{Biorthogonal systems in {B}anach spaces}, CMS Books in
  Mathematics/Ouvrages de Math\'ematiques de la SMC, 26, Springer, New York,
  2008. 

\bibitem{hei}
M.~Heili{\"o}, \emph{Weakly summable measures in {B}anach spaces}, Ann. Acad.
  Sci. Fenn. Ser. A I Math. Dissertationes (1988), no.~66. 

\bibitem{Kunen68}
K.~Kunen, \emph{Inaccessibility properties of cardinals}, ProQuest LLC, Ann
  Arbor, MI, 1968. Thesis (Ph.D.) Stanford University.

\bibitem{mar-pol-1}
W.~Marciszewski and R.~Pol, \emph{On {B}anach spaces whose norm-open sets are
  {$F_\sigma$}-sets in the weak topology}, J. Math. Anal. Appl. \textbf{350}
  (2009), no.~2, 708--722.

\bibitem{mar-pol-2}
W.~Marciszewski and R.~Pol, \emph{On some problems concerning {B}orel structures in function
  spaces}, Rev.\ Real Acad.\ Ciencias Exactas, Fisicas y Naturales. Serie A.
  Matematicas \textbf{104} (2010), no.~2, 327--335.

\bibitem{whe2}
E.~Marczewski, R.~Sikorski, \emph{Measures in non-separable metric spaces}. Colloq. Math. \textbf{1} (1948). 133--139.

\bibitem{mer-J}
S.~Mercourakis, \emph{Some remarks on countably determined measures and uniform
  distribution of sequences}, Monatsh. Math. \textbf{121} (1996), no.~1-2,
  79--111. 

\bibitem{Plachky92}
D. Plachky, \emph{Some measure theoretical characterizations of separability of metric spaces},
Arch. Math. \textbf{58} (1992), 366--367.

\bibitem{ple5}
G.~Plebanek, \emph{On the space of continuous functions on a dyadic set},
  Mathematika \textbf{38} (1991), no.~1, 42--49.

\bibitem{rod9}
J.~Rodr{\'{\i}}guez, \emph{Weak {B}aire measurability of the balls in a
  {B}anach space}, Studia Math. \textbf{185} (2008), no.~2, 169--176.

\bibitem{rod-ver}
J.~Rodr{\'{\i}}guez and G.~Vera, \emph{Uniqueness of measure extensions in
  {B}anach spaces}, Studia Math. \textbf{175} (2006), no.~2, 139--155.

\bibitem{tal9}
M.~Talagrand, \emph{Comparaison des boreliens d'un espace de {B}anach pour les
  topologies fortes et faibles}, Indiana Univ. Math. J. \textbf{27} (1978),
  no.~6, 1001--1004. 

\bibitem{tal10}
M.~Talagrand, \emph{Est-ce que {$l^\infty$} est un espace measurable?}, Bull. Sci.
  Math. \textbf{103} (1979), 255--258.

\bibitem{tal}
M.~Talagrand, \emph{Pettis integral and measure theory}, Mem. Amer. Math. Soc.
  \textbf{51} (1984), no.~307, ix+224. 

\bibitem{Todorcevic}
S. Todorcevic, \emph{Embedding function spaces into $\ell_\infty/c_0$}, J. Math. Anal. App.
\textbf{384} (2011), no.~2, 246-251

\bibitem{Ulam}
M. S. Ulam, \emph{Probl\`{e}mes 74}, Fund. Math. \textbf{30} (1938), 365.


\end{thebibliography}

\end{document}